%% file: iAM.tex
\documentclass[12pt]{article}
\usepackage[top=1.2in, bottom=1.2in, left=1in, right=1in]{geometry}
\usepackage{amsmath,amssymb,amsfonts,amsthm,xcolor}
 \pdfoutput=1

\usepackage{ifpdf}
\usepackage[pdfpagelabels]{hyperref}
\usepackage{booktabs}
\usepackage{ifthen}
\usepackage{xspace}
\usepackage[pdftex]{graphicx}
\usepackage{multirow}
\usepackage{colortbl}
\usepackage{mathrsfs}
\usepackage{enumerate}
\usepackage{enumitem}
\usepackage{subfigure}

\usepackage{tikz}

\usepackage{fancyhdr}
\fancypagestyle{plain}{ %
  \fancyhf{} 
  
}
\makeatletter
\newcommand{\@KEYWORDS}{\@empty}
\newcommand{\KEYWORDS}[1]{\renewcommand{\@KEYWORDS}{#1}}
\newcommand{\makekeywords}{\begin{flushleft}{\footnotesize \textbf{\textit{Keywords ---}} \textit{\@KEYWORDS}}\end{flushleft}}
\newcommand{\@AMS}{\@empty}
\newcommand{\AMS}[1]{\renewcommand{\@AMS}{#1}}
\newcommand{\makeams}{\begin{flushleft}{\footnotesize \textbf{\textit{AMS (2010) ---}} \textbf{\@AMS}}\end{flushleft}}
\makeatother

\usepackage[font=small]{caption}
\input{mycommands.tex}

\renewcommand{\k}{k}
\newcommand{\kp}{{k+1}}
\newcommand{\kpp}{{k+2}}
\newcommand{\kppp}{{k+3}}
\newcommand{\km}{{k-1}}

\definecolor{miared}{rgb}{0.7,0.15,0}
\definecolor{miablue}{rgb}{0.058,0.047,.596} 
\definecolor{mygreen}{rgb}{0.058,.596,0.047}

\newcommand{\F}{F}
\newcommand{\FF}{\mathcal F}

\usepackage{pgfplots}
\pgfplotsset{
  tick label style={font=\footnotesize},
  label style={font=\footnotesize},
  legend style={font=\footnotesize}
}

\newcommand{\dimN}{N}

\newcommand{\dimM}{M}
\newcommand{\dimR}{R}
\newcommand{\dimD}{D}

\newcommand{\lip}{\mathrm{lip}}
\newcommand{\floc}{T}
\newcommand{\rge}{\mathrm{rge}\,}

\newcommand{\KL}{Kurdyka--{\L}ojasiewicz\xspace}

\newcommand{\ivalcc}[1]{[#1]}
\newcommand{\ivalco}[1]{[#1[}
\newcommand{\ivaloo}[1]{]#1[}
\newcommand{\ivaloc}[1]{]#1]}
\newcommand{\ipianodelta}{\kappa}
\newcommand{\setsep}{\,:\,}

\colorlet{changeColor}{red!90!black}

\graphicspath{{./figures/}}
\def\codePATH{./figures}

\title{Local Convergence of the Heavy-ball Method and iPiano for Non-convex Optimization}

\author{Peter Ochs\\
  Mathematical Optimization Group\\
  Saarland University,
  Germany \\
  {\tt\small ochs@math.uni-sb.de}
}

\AMS{90C26, 
     90C30, 
     65K05} 

\begin{document}

\maketitle

\begin{abstract}
A local convergence result for an abstract descent method is proved. The sequence of iterates is attracted by a local (or global) minimum, stays in its neighborhood and converges within this neighborhood. This result allows algorithms to exploit local properties of the objective function. In particular, the abstract theory in this paper applies to the inertial forward--backward splitting method: iPiano---a generalization of the Heavy-ball method. Moreover,  it reveals an equivalence between iPiano and inertial averaged/alternating proximal minimization and projection methods. Key for this equivalence is the attraction to a local minimum within a neighborhood and the fact that, for a prox-regular function, the gradient of the Moreau envelope is locally Lipschitz continuous and expressible in terms of the proximal mapping. In a numerical feasibility problem, the inertial alternating projection method significantly outperforms its non-inertial variants.
\end{abstract}

\makekeywords
\makeams

\section{Introduction}

In non-convex optimization, we often content ourselves with local properties of the objective function. Exploiting local information such as smoothness or prox-regularity around the optimum yields a local convergence theory. Local convergence rates can be obtained or iterative optimization algorithms can be designed that depend on properties that are available only locally around a local optimum. For revealing such results, it is crucial that the generated sequence, once entered such a neighborhood of a local optimum, stays within this neighborhood and converges to a limit point in the same neighborhood. 

As an illustrative example, suppose a point close to a local minimizer can be found by a global method, for example, by exhaustive search. In a neighborhood of the local minimizer, we can switch to a more efficient local algorithm. The local attraction of the local minimum assures that the generated sequence of iterates stays in this neighborhood, i.e. the sequence does not escape to a different local minimum, and there is no need to switch back to the (slow) global method, exhaustive search.

An important example of local properties, which we are going to exploit in this paper, is the fact that the Moreau envelope of a prox-regular function is locally well-defined and its gradient is Lipschitz continuous and expressible using the proximal mapping---a result that is well known for convex functions. Locally, this result can be applied to gradient-based iterative methods for minimizing objective functions that involve a Moreau envelope of a function. We pursue this idea for the Heavy-ball method \cite{Polyak64,ZK93} and iPiano \cite{OCBP14,Ochs15} (inertial version of forward--backward splitting) and obtain new algorithms for non-convex optimization such as inertial alternating/averaged proximal minimization or projection methods. The convergence result of the Heavy-ball method and iPiano translates directly to these new methods in the non-convex setting. The fact that a wide class of functions is prox-regular extends the applicability of these inertial methods significantly.\\

Prox-regularity was introduced in \cite{PR96} and comprises primal-lower-nice (introduced by Poliquin \cite{Poliquin91}), lower-$\mathcal C^2$, strongly amenable (see for instance \cite{Rock98}), and proper lower semi-continuous convex functions. It is known that prox-regular functions (locally) share some favorable properties of convex functions, e.g. the formula for the gradient of a Moreau envelope. Indeed a function is prox-regular if and only if there exists an ($f$-attentive) localization of the subgradient mapping that is monotone up to a multiple of the identity mapping \cite{PR96}. 
%
In \cite{ABS13}, prox-regularity is key to prove local convergence of the averaged projection method using the gradient descent method, which is a result that has motivated this paper.

The convergence proof of the gradient method in \cite{ABS13} follows a general paradigm that is currently actively used for the convergence theory in non-convex optimization. The key is the so-called \KL (KL) property \cite{Kurd98,Loj63,Loj93,BDL06,BDLS07}, which is known to be satisfied by semi-algebraic \cite{BCR98}, globally subanalytic functions \cite{BDL06b}, or more generally, functions that are definable in an $\o$-minimal structure \cite{BDLS07,Dries98}. 
Global convergence of the full sequence generated by an abstract algorithm to a stationary point is proved for functions with the KL property. The algorithm is abstract in the sense that the generated sequence is assumed to satisfy a \emph{sufficient descent condition}, a \emph{relative error condition}, and a \emph{continuity condition}, however no generation process is specified. 

The following works have also shown global convergence using the KL property or earlier versions thereof. The gradient descent method is considered in \cite{AMA05,ABS13}, the proximal algorithm is analyzed in \cite{AB09,ABS13,BDLM10,BS15}, and the non-smooth subgradient method in \cite{Noll13,Hosseini15}. Convergence of forward--backward splitting (proximal gradient algorithm) is proved in \cite{ABS13}. Extensions to a variable metric are studied in \cite{CPR14}, and in \cite{BLPPR16} with line search. A block coordinate descent version is considered in \cite{XY17} and a block coordinate variable metric method in \cite{CPR16}. A flexible relative error handling of forward--backward splitting and a non-smooth version  of the Levenberg--Marquardt algorithm is explored in \cite{FGP14}. For proximal alternating minimization, we refer to \cite{ABRS10} for an early convergence result of the iterates, and to \cite{BST14} for proximal alternating linearized minimization. 

Inertial variants of these algorithms have also been examined. \cite{OCBP14} establishes convergence of an inertial forward--backward splitting algorithm, called iPiano. \cite{OCBP14} assumes the non-smooth part of the objective to be convex, whereas \cite{Ochs15} and \cite{BCL15} prove convergence in the full non-convex setting, i.e. when the algorithm is applied to minimizing the sum of a smooth non-convex function with Lipschitz gradient and a proper lower semi-continuous function. An extension to an inertial block coordinate variable metric version was studied in \cite{Ochs16}. Bregman proximity functions are considered in \cite{BCL15}. A similar method was considered in \cite{BC16} by the same authors. The convergence of a generic multi-step method is proved in \cite{LFP16} (see also \cite{JM16}). A slightly weakened formulation of the popular accelerated proximal gradient algorithm from convex optimization was analyzed in \cite{LL15}. Another fruitful concept from convex optimization is that of composite objective functions involving linear operators. This problem is approached in \cite{STP17,LP15b}. Key for the convergence results is usually a decrease condition on the objective function or an upper bound of the objective. The Lyapunov-type idea is studied in \cite{LP15a,LP16a,LP15b}. Convergence of the abstract principle of majorization minimization methods was also analyzed in a KL framework \cite{ODBP15,BP16}. 

The global convergence theory of an unbounded memory multi-step method was proposed in \cite{LFP16}. Local convergence was analyzed under the additional partial smoothness assumption. In particular local linear convergence of the iterates is established. Although the fruitful concept of partial smoothness is very interesting, in this paper, we focus on convergence results that can be inferred directly from the KL property. In the general abstract setting, local convergence rates were analyzed in \cite{FGP14,LP17} and for inertial methods in \cite{LP17,JM16}. More specific local convergence rates can be found in \cite{AB09,MP10,ABRS10,XY13,BST14,CPR16}.


%
While the abstract concept in \cite{ABS13} can be used to prove global convergence in the non-convex setting for the gradient descent method, forward--backward splitting, and several other algorithms, it seems to be limited to single-step methods. Therefore, \cite{OCBP14} proved a slightly different result for abstract descent methods, which is applicable to multi-step methods, such as the Heavy-ball method and iPiano. In \cite{Ochs16}, an abstract convergence result is proved that unifies \cite{ABS13,FGP14,OCBP14,Ochs15}.

\paragraph{Contribution.} In this paper, we develop the \emph{local convergence theory} for the abstract setting in \cite{OCBP14}, in analogy to the local theory in \cite{ABS13}. Our local convergence result shows that, for multi-step methods such as the \emph{Heavy-ball method} or \emph{iPiano},
{a sequence that is initialized close enough to a local minimizer
\begin{itemize} 
  \item \emph{stays in a neighborhood of the local minimum} and 
  \item \emph{converges to a local minimizer} instead of a stationary point.
 \end{itemize}%
  This result  allows us to apply the formula for the gradient of the Moreau envelope of a prox-regular function to all iterates, which has far-reaching consequences and has not been explored algorithmically before. We obtain \emph{several new algorithms} for non-convex optimization problems. Conceptionally the algorithms are known from the convex setting or from their non-inertial versions, however \emph{there are no guarantees for the inertial versions in the non-convex setting}.
%
\begin{itemize}
  \item The Heavy-ball method applied to the sum of distance functions to prox-regular sets (resp. the sum of Moreau envelopes of prox-regular functions) coincides with the inertial averaged projection method (resp. the inertial averaged proximal minimization) for these prox-regular sets  (resp. functions). 
  \item iPiano applied to the sum of the distance function to a prox-regular set (resp. the Moreau envelope of a prox-regular function) and a simple non-convex set (resp. function) leads to the inertial alternating projection method (resp. inertial alternating proximal minimization) for these two sets (resp. functions).
\end{itemize}
Of course, these algorithms are only efficient when the associated proximal mappings or projections are simple  (efficient to evaluate). Beyond these local results, we provide \emph{global convergence guarantees} for the following methods:
\begin{itemize}
  \item The (relaxed) alternating projection method for the feasibility problem of a convex set and a non-convex set.
  \item An inertial version of the alternating projection method (iPiano applied to the distance function to a convex set over a non-convex constraint set).
  \item An inertial version of alternating proximal minimization (iPiano applied to the sum of the Moreau envelope of a convex function and a non-convex function). 
\end{itemize}

Moreover, we transfer \emph{local convergence rates} depending on the KL exponent of the involved functions to the methods listed above. This result builds on a recent classification of local convergence rates depending on the KL exponent from \cite{LP17,JM16} (which extends results from \cite{FGP14}).

\paragraph{Outline.}
Section~\ref{sec:prelim} introduces the notation and definitions that are used in this paper.  In Section~\ref{subsec:abstract-conv-global} the conditions for global convergence of abstract descent methods \cite{OCBP14,Ochs15} are recapitulated. The main result for abstract descent methods, the \emph{attraction of local (or global) minima}, is developed and proved in Section~\ref{subsec:abstract-conv-local}. Then, the abstract local convergence results are \emph{verified for iPiano} (hence the Heavy-ball method) in Section~\ref{sec:review-of-iPiano}. The \emph{equivalence} to inertial averaged/alternating minimization/projection methods is analyzed in Section~\ref{sec:alt-avg-prox-min}. Section~\ref{sec:experiment-feasibility} shows a numerical example of a \emph{feasibility problem}.

\section{Preliminaries} \label{sec:prelim}

Throughout this paper, we will always work in a finite dimensional Euclidean vector space $\R^\dimN$ of dimension $\dimN\in\N$, where $\N:=\{1,2,\ldots\}$. The vector space is equipped with the standard Euclidean norm $\vnorm\cdot$ that is induced by the standard Euclidean inner product $\vnorm\cdot = \sqrt{\scal\cdot\cdot}$. We denote by $\ball \eps (\bar x):=\set{x\in\R^\dimN\setsep \vnorm{x-\bar x}\leq \eps}$ the ball of radius $\eps>0$ around $\bar x\in \R^\dimN$. 

As usual, we consider extended real-valued functions $\map f{\R^\dimN}{\eR}$, where $\eR:=\R\cup\set{+\infty}$, which are defined on the whole space with \emph{domain} given by $\dom f := \set{x\in\R^\dimN \setsep f(x)< +\infty}$. A function is called \emph{proper} if $\dom f \neq \emptyset$. We define the \emph{epigraph} of the function $f$ as $\epi f:= \set{(x,\mu)\in\R^{\dimN+1} \setsep \mu \geq f(x)}$. The \emph{indicator function} $\ind C$ of a set $C\subset\R^N$ is defined by $\ind C(x)=0$, if $x\in C$, and $\ind C(x)=+\infty$, otherwise. A \emph{set-valued mapping} $\smap{T}{\R^N}{\R^M}$, with $M,N\in\N$, is defined by its \emph{graph} $\Graph T:=\set{(x,v)\in \R^N\times\R^M \setsep v\in T(x)}$. The range of a set-valued mapping is defined as $\rge T:= \bigcup_{x\in \R^\dimN} T(x)$.

A key concept in optimization and variational analysis is that of Lipschitz continuity. Sometimes, also the term strict continuity is used, which we define as in \cite{Rock98}:
\begin{DEF}[strict continuity {\cite[Definition 9.1]{Rock98}}]
  A single-valued mapping $\map F{D}{\R^\dimM}$ defined on $D\subset \R^\dimN$ is \emph{strictly continuous} at $\bar x\in D$ if the value 
  \[
     \lip F (\bar x) := \limsup_{\substack{x,x^\prime \to \bar x\\ x\neq x^\prime}} \frac{\abs{F(x^\prime) - F(x)}}{\abs{x^\prime - x}} 
  \]
  is finite and $\lip F(\bar x)$ is the \emph{Lipschitz modulus} of $F$ at $\bar x$. This is the same as saying $F$ is locally Lipschitz continuous at $\bar x$ on $D$.
\end{DEF}
For convenience, we introduce \emph{$f$-attentive convergence}: A sequence $\seq[\k\in\N]{x^\k}$ is said to \emph{$f$-converge} to $\bar x$ if 
$(x^\k,f(x^\k))\to (\bar x , f(\bar x))$ as $\k\to \infty$, 
and we write $x^\k\fto \bar x$. 
\begin{DEF}[subdifferentials {\cite[Definition 8.3]{Rock98}}]
The \emph{Fr\'echet subdifferential} of $f$ at $\bar x \in\dom f$ is the set $\rpartial f(\bar x)$ of elements $v \in \R^\dimN$ such that
\[
  \liminf_{\substack{x\to \bar x\\ x\neq \bar x}} \frac{f(x) - f(\bar x) - \scal{v}{x-\bar x}}{\vnorm{x-\bar x}} \geq 0  \,.
\]
For $\bar x\not\in \dom f$, we set $\rpartial f(\bar x) = \emptyset$. 
The so-called \emph{(limiting) subdifferential} of $f$ at $\bar x\in\dom f$ is defined by
\[
  \partial f(\bar x) := \set{v\in \R^\dimN\setsep \exists\, x^n \fto \bar x,\;v^n\in \rpartial f(x^n),\;v^n \to v} \,,
\]
and $\partial f(\bar x) = \emptyset$ for $\bar x \not\in \dom f$. 
\end{DEF}
A point $\bar x\in \dom f$ for which $0\in \partial f(\bar x)$ is a called a \emph{critical point}. As a direct consequence of the definition of the limiting subdifferential, we have the following closedness property at any $\bar x\in \dom f$: 
\[
  x^\k \fto \bar x,\ v^\k\to \bar v,\ \text{and for all } \k\in\N\colon v^\k \in \partial f(x^\k)\quad \Longrightarrow\quad  \bar v\in \partial f(\bar x) \,.
\]
\begin{DEF}[Moreau envelope and proximal mapping {\cite[Definition 1.22]{Rock98}}]
For a function $\map {f}{\R^\dimN}{\eR}$ and $\lambda>0$, we define the \emph{Moreau envelope}
\[
  \menv \lambda f (x) := \inf_{w\in\R^\dimN}\, f(w) + \frac{1}{2\lambda}\vnorm{w-x}^2 \,,
\]
and the \emph{proximal mapping}
\[
  \prox \lambda f (x) := \arg\min_{w\in\R^\dimN}\, f(w) + \frac{1}{2\lambda}\vnorm{w-x}^2 \,.
\]
\end{DEF}
For a general function $f$ it might happen that $\menv \lambda f(x)$ takes the values $- \infty$ and the proximal mapping is empty, i.e. $\prox \lambda f (x) = \emptyset$. Therefore, the analysis of the Moreau envelope is usually coupled with the following property.
\begin{DEF}[prox-boundedness {\cite[Definition 1.23]{Rock98}}]
  A function $\map f{\R^\dimN}{\eR}$ is \emph{prox-bounded}, if there exists $\lambda >0$ such that $\menv \lambda f (x)> -\infty$ for some $x\in \R^\dimN$. The supremum of the set of all such $\lambda$ is the \emph{threshold} $\lambda_f$ of prox-boundedness for $f$.
\end{DEF}
In this paper, we focus on so-called prox-regular functions.  These functions have many favorable properties locally, which otherwise only convex functions exhibit. 
\begin{DEF}[prox-regularity of functions, {\cite[Definition 13.27]{Rock98}}]
A function $\map f{\R^\dimN}{\eR}$ is \emph{prox-regular} at $\bar x$ for $\bar v$ if $f$ is finite and locally lsc at $\bar x$ with $\bar v \in \partial f(\bar x)$, and there exists $\eps>0$ and $\lambda>0$ such that
\begin{gather*}
  f(x^\prime) \geq f(x) + \scal{v}{x^\prime - x} - \frac 1{2\lambda} \vnorm{x^\prime - x}^2 \quad \forall x^\prime \in \ball\eps (\bar x) \\
  \text{when}\ v\in \partial f(x),\ \vnorm{v-\bar v } < \eps,\ \vnorm{x-\bar x}<\eps,\ f(x) < f(\bar x) + \eps\,.
\end{gather*}
When this holds for all $\bar v\in \partial f(\bar x)$, $f$ is said to be prox-regular at $\bar x$. 
\end{DEF}
The largest value $\lambda>0$ for which this property holds is called the \emph{modulus of prox-regularity at $\bar x$}.
\begin{DEF}[prox-regularity of sets, {\cite[Exercise 13.31]{Rock98}}]
A set $C$ is \emph{prox-regular} at $\bar x$ for $\bar v$ when the indicator function $\ind C$ of the set $C$ is prox-regular at $\bar x$ for $\bar v$. It is called \emph{prox-regular} at $\bar x$, when this is true for all $\bar v\in \partial \ind C(\bar x)$.
\end{DEF}

To observe that most functions in practice are prox-regular, we provide several examples.
\begin{EX} A function $\map{f}{\R^\dimN}{\eR}$ is prox-regular if, for example\footnote{For the exact statements, we provide accurate references.}, the function $f$ is
\begin{itemize}
  \item proper lower semi-continuous (lsc) convex \cite[Example 13.30]{Rock98},
  \item locally representable in the form $f=g-\rho\vnorm\cdot^2$ with $g$ being finite ($g<+\infty$) convex and $\rho>0$ \cite[Theorem 10.33]{Rock98},
  \item strongly amenable \cite[Definition 10.23, Proposition 13.32]{Rock98} (e.g. $\mathcal C^2$-functions, functions of the form $g\circ F$ with $F$ being $\mathcal C^2$ and $g$ being proper lsc convex, the maximum of $\mathcal C^2$-function),
  \item lower-$\mathcal C^2$ \cite[Definition 10.29, Proposition 13.33]{Rock98} (functions of the form $\max_{t\in T} f(x,t)$, where the zeroth, first and second derivative of $\map{f}{\R^\dimN\times T}{\R}$ w.r.t. to the first block of coordinates are continuous and $T$ is a compact space),
  \item a $\mathcal C^2$-perturbation of a prox-regular function \cite[Exercise 13.35]{Rock98},
  \item an indicator function of a closed convex set or of a strongly amenable set \cite[Definition 10.23]{Rock98}, 
  \item the Moreau envelope of a prox-regular prox-bounded function \cite[Proposition 13.37 ana 13.34]{Rock98} (e.g. the distance function of a prox-regular set), or
  \item the indicator function of a closed set $C$ and the distance function w.r.t. $C$ is continuously differentiable on $C\smallsetminus U$ for some open neighborhood $U$ \cite[Theorem 1.3]{PRT00}.
  \item For examples of prox-regular spectral functions, we refer to \cite{DLMS08}.
\end{itemize}
\end{EX}
\begin{EX}[Imaging problems]
  Several problems (image denoising, deblurring/deconvolution, zooming, depth map fusion, etc.) may be modeled as an optimization problem of the following form
  \[
    \min_{x\in \R^N}\, \frac 12\vnorm{Ax - b}^2 + \sum_{i=1}^{N} \varphi( \sqrt{(Dx)_i^2 + (Dx)_{i+N}^2})\,,
  \]
  where $A\in \R^{M\times N}$ (e.g. blurr operator), $b\in \R^M$ (e.g. blurry input image), $D\in \R^{2N\times N}$ (e.g. finite differences) with a continuous non-decreasing function $\map{\varphi}{\R_+}{\R_+}$. The objective is prox-regular, for example, under the following conditions: $\varphi$ is convex and non-decreasing; $\varphi(t)=t$ (TV-regularization); $\varphi(t) = \log(1+t^2)$ (student-t regularization); $\varphi(t) = \log(1+\abs{t})$ (at $0$ where it is not $\mathcal C^2$, the power series of $\log$ shows that $\log(1+\abs{t})-\abs{t}\in\mathcal C^2$, hence $\varphi$ is a $\mathcal C^2$-perturbation of a convex function).
\end{EX}
\begin{EX}[Support Vector Machine]
The goal to find a linear decision function may be formulated as the following optimization problem
\[
    \min_{w\in\R^N,\, b\in\R}\, \sum_{i=1}^M \mathscr L(\scal{w}{z_i}+b, y_i) + \varphi(w) \,, 
\]
where, for $i=1,\ldots,M$, $(z_i,y_i)\in \R^N\times\set{\pm1}$ is the training set, $\mathscr L$ is a loss function, and $\varphi$ a regularizer. Examples are the hinge loss $\mathscr L(\bar y_i,y_i) = \max(0, 1- \bar y_iy_i)$ (which is a maximum of $\mathcal C^2$-functions), the squared hinge loss, the logistic loss $\mathscr L(\bar y_i,y_i) = \log(1+e^{-\bar y_i y_i})$ (which are $\mathcal C^2$ function), etc. Prox-regular regularization functions $\varphi$ are, for example, the squared $\ell_2$-norm $\vnorm{x}^2$, or more in general $p$-norms $\norm[p]{x}^p=\sum_{i=1}^N \abs{x_i}^p$ with $p>0$ ($x_i\mapsto \abs{x_i}^p$ is $\mathcal C^2$ on $\R\smallsetminus\set0$ and obviously prox-regular at $\bar x=0$).
\end{EX}

For the proof of the Lipschitz property of the Moreau envelope, it will be helpful to consider a so-called localization. A \emph{localization} of $\partial f$ around $(\bar x,\bar v)$ is a mapping $\smap T{\R^\dimN}{\R^\dimN}$ whose graph is obtained by intersecting $\Graph \partial f$ with some neighborhood of $(\bar x, \bar v)$, i.e. $\Graph T = \Graph \partial f \cap U$ for a neighborhood $U$ of $(\bar x, \bar v)$. We talk about an \emph{$f$-attentive localization} when $\Graph T = \set{(x,v)\in \Graph \partial f\setsep (x,v) \in U\ \text{and}\ f(x) \in V}$ for a neighborhood $U$ of $(\bar x, \bar v)$ and a neighborhood $V$ of $f(\bar x)$.\\

Finally, the convergence result we build on is only valid for functions that have the KL property at a certain point of interest. This property is shared for example by semi-algebraic functions, globally subanalytic functions, or, more generally, functions definable in an $\o$-minimal structure. For details, we refer to \cite{BDL06,BDLS07}.
\begin{DEF}[\KL property / KL property {\cite{ABS13}}]\label{def:KL-property}
Let $\map f {\R^\dimN}{\eR}$ be an extended real valued function and let $\bar x\in\dom\partial f$. If there exists $\eta\in\ivalcc{0,\infty}$, a neighborhood $U$ of $\bar x$ and a continuous concave function $\map{\phi}{\ivalco{0,\eta}}{\R_+}$ such that 
\[
  \phi(0)=0,\quad  \phi\in\spC 1( 0,\eta ),\quad\text{and}\quad \phi^\prime(s)>0\text{ for all }s\in \ivaloo{0,\eta}\,,
\]
and for all $x\in U\cap \ivalcc{f(\bar x) < f(x) < f(\bar x) + \eta}$ the \KL inequality
\begin{equation}\label{eq:KL-ineq}
  \phi^\prime(f(x)-f(\bar x)) \norm[-]{\partial f(x)} \geq 1
\end{equation}
holds, then the function has the \KL property at $\bar x$, where $\norm[-]{\partial f(x)}:= \inf_{v\in \partial f(x)}\vnorm{v}$ is the \emph{non-smooth slope} (note: $\inf \emptyset := +\infty$).

If, additionally, the function is lsc and the property holds for each point in $\dom \partial f$, then $f$ is called \KL function.
\end{DEF}
If $f$ is closed and semi-algebraic, it is well-known \cite{Kurd98,BDLS07} that $f$ has the KL property at any point in $\dom\partial f$, and the \emph{desingularization function} $\phi$ in Definition~\ref{def:KL-property} has the form $\phi(s) = \frac{c}{\theta}s^\theta$ for $\theta\in \ivaloc{0,1}$ and some constant $c>0$. The parameter $\theta$ is known as the \emph{KL exponent}.

\section{Abstract Convergence Result for KL Functions} \label{sec:abstract-convergence}

In this section, we establish a local convergence result for abstract descent methods , i.e., the method is characterized by properties \ref{ass:descent}, \ref{ass:error}, \ref{ass:cont} (see below) instead of a specific update rule. The local convergence result is inspired by a global convergence result proved in \cite{OCBP14} for KL functions (see Theorem~\ref{thm:abstr-global-conv}), which itself is motivated by a slightly different result in \cite{ABS13}. The abstract setting in \cite{ABS13}, can be used to prove global and local convergence of gradient descent, proximal gradient descent and other (single-step) methods. However, it does not apply directly to inertial variants of these methods. Therefore, in this section, we prove the required adaptation of the framework in \cite{ABS13} to the one in \cite{OCBP14}. We obtain a local convergence theory that also applies to the Heavy-ball method and iPiano (see Section~\ref{sec:review-of-iPiano}). 

\subsection{Global Convergence Results} \label{subsec:abstract-conv-global}
The convergence result in \cite{OCBP14} is based on the following three abstract conditions for a sequence $(z^\k)_{\k\in\N}:=(x^\k,x^{\km})_{\k\in\N}$ in $\R^{2N}$, $x^\k\in\R^N$, $x^{-1}\in\R^N$. Fix two positive constants $a>0$ and $b>0$ and consider a proper lower semi-continuous (lsc) function $\map{\FF}{\R^{2N}}{\eR}$. Then, the conditions for $(z^\k)_{\k\in\N}$ are as follows:
\begin{enumerate}[label=(H\arabic*),ref=(H\arabic*)]
\item\label{ass:descent} For each $\k\in \N$, it holds that
\[
    \FF(z^{\kp}) + a \vnorm{x^\k - x^\km}^2 \leq \FF(z^\k)\,.
\]
\item\label{ass:error} For each $\k\in\N$, there exists $w^{\kp}\in\partial\FF(z^{\kp})$ such that
\[
    \vnorm{w^{\kp}} \leq \frac b2 (\vnorm{x^\k - x^\km} + \vnorm{x^\kp - x^\k})\,.
\]
\item\label{ass:cont} There exists a subsequence $(z^{\k_j})_{j\in\N}$ such that
\[
    z^{\k_j}\to \tilde z \quad \text{and}\quad \FF(z^{\k_j}) \to \FF(\tilde z)\,,\qquad \text{as } j\to\infty\,.
\]
\end{enumerate}

\begin{THM}[abstract global convergence, {\cite[Theorem 3.7]{OCBP14}}] \label{thm:abstr-global-conv}
Let $\seq[\k\in\N]{z^\k} = \seq[\k\in\N]{x^\k,x^{\km}}$ be a sequence that satisfies \ref{ass:descent}, \ref{ass:error}, and \ref{ass:cont} for a proper lsc function $\map{\FF}{\R^{2N}}\eR$ which has the KL property at the cluster point $\tilde z$ specified in \ref{ass:cont}.\\
Then, the sequence $\seq[\k\in\N]{x^\k}$ has finite length, i.e.
\begin{equation} \label{eq:thm:abstr-global-conv:finite-length-prop}
  \sum_{\k=1}^{\infty} \vnorm{x^\k-x^\km} < + \infty \,,
\end{equation}
and converges to $\bar z = \tilde z$ where $\bar z=(\bar x,\bar x)$ is a critical point of $\FF$.
\end{THM}
\begin{REM}
  In view of the proof of this statement, it is clear that the same result can be established when \ref{ass:descent} is replaced by
  $
    \FF(z^{\kp}) + a \vnorm{x^\kp - x^\k}^2 \leq \FF(z^\k) \,.
  $
\end{REM}

\subsection{Local Convergence Results} \label{subsec:abstract-conv-local}

The upcoming local convergence result shows that, once entered a region of attraction (around a local minimizer), all iterates of a sequence $\seq[\k\in\N]{z^\k}$ satisfying \ref{ass:descent}, \ref{ass:error} and the following growth condition \ref{ass:local} stay in a neighborhood of this minimum and converge to a minimizer in the same neighborhood (not just a stationary point). For the convergence to a global minimizer, the growth condition \ref{ass:local} is not required.  \\

In the following, for $z\in \R^{2N}$ we denote by $z_1,z_2 \in \R^N$ the first and second block of coordinates, i.e. $z=(z_1,z_2)$. The same holds for other vectors in $\R^{2N}$.

\begin{enumerate}[label=(H\arabic*),ref=(H\arabic*)]
  \setcounter{enumi}{3}
  \item\label{ass:local} Fix $z^*\in \R^N$. For any $\delta >0$ there exist $0<\rho <\delta$ and $\nu>0$ such that
  \[
    z\in \ball \rho (z^*)\,,\ \FF(z) < \FF(z^*) + \nu\,,\ y_2\not\in \ball \delta (z_2^*) \ \Rightarrow \ 
    \FF(z) < \FF(y) + \frac a4 \vnorm{z_2 - y_2}^2 
  \]
  where $a$ is the same as in \ref{ass:descent}--\ref{ass:cont}.
\end{enumerate}
A simple condition that implies \ref{ass:local} is provided by the following lemma: 
\begin{LEM} \label{lem:local-growth}
  Let $\map{\FF}{\R^{2N}}\eR$ be a proper lsc function and $z^*=(x^*,x^*)\in \dom\FF$ a local minimizer of $\FF$. Suppose, for any $\delta >0$, $\FF$ satisfies the growth condition 
  \[
    \FF(y) \geq \FF(z^*) - \frac {a}{16} \abs{y_2 - z^*_2}^2 \quad \forall y\in \R^{2N}, y_2\not\in\ball \delta(z_2^*) \,.
  \]
  Then, $\FF$ satisfies \ref{ass:local}.
\end{LEM}
\begin{proof}
  Let $\delta > \rho$ and $\nu $ be positive numbers. For $y=(y_1,y_2)\in \R^{2\dimN}$ with $y_2\not\in \ball \delta (z_2^*)$ and $z=(z_1,z_2)\in \ball \rho (z^*)$ such that $\FF(z) < \FF(z^*) + \nu$, we make the following estimation:  
  \[
    \begin{split}
      \FF(y) \geq&\ \FF(z^*) - \frac a{16} \vnorm{y_2 - z_2^*}^2 \\
             >&\ \FF(z) - \nu - \frac a8 \vnorm{y_2 - z_2^*}^2 + \frac a{16} \vnorm{y_2 - z_2^*}^2 \\
             \geq&\ \FF(z) - \nu - \frac a4 \vnorm{y_2 - z_2}^2 - \frac a4 \vnorm{z_2 - z_2^*}^2 + \frac a{16} \vnorm{y_2 - z_2^*}^2 \\
             \geq&\ \FF(z) - \frac a4 \vnorm{y_2 - z_2}^2 + (-\nu - \frac a4 \rho^2 + \frac a{16} \delta^2) \,.
    \end{split}
  \]
  For sufficiently small $\nu$ and $\rho$ the term in the parenthesis becomes positive, which implies \ref{ass:local}.
\end{proof}
We need another preparatory lemma, which is proved in \cite{OCBP14} 
\begin{LEM}[{\cite[Lemma 3.5]{OCBP14}}]\label{lem:main-theorem-convergence}
Let $\map{\FF}{\R^{2N}}{\eR}$ be a proper lsc function which satisfies the \KL property at some point $z^*=(z_1^*,z_2^*)\in \R^{2N}$. Denote by $U$, $\eta$ and $\map{\phi}{\ivalco{0,\eta}}{\R_+}$ the objects appearing in Definition~\ref{def:KL-property} of the KL property at $z^*$. Let $\sigma, \rho >0$ be such that $\ball\sigma (z^*)\subset U$ with $\rho \in \ivaloo{0,\sigma}$.

Furthermore, let $(z^\k)_{\k\in\N}=(x^\k,x^{\km})_{\k\in\N}$ be a sequence satisfying \ref{ass:descent}, \ref{ass:error}, and 
\begin{equation}\label{eq:mt-eqA}
  \forall \k\in\N:\quad z^\k\in \ball\rho(z^*) \Rightarrow z^{\kp}\in \ball\sigma(z^*)\text{ with } \FF(z^{\kp}),\FF(z^{\k+2}) \geq \FF(z^*) \,.
\end{equation}
Moreover, the initial point $z^0=(x^0,x^{-1})$ is such that $\FF(z^*) \leq \FF(z^0) < \FF(z^*) + \eta$ and 
\begin{equation}\label{eq:mt-eqD}
\vnorm{x^*-x^0} + \sqrt{\frac {\FF(z^0) - \FF(z^*)}a} + \frac ba \phi(\FF(z^0) - \FF(z^*)) < \frac \rho 2 \,.
\end{equation}
Then, the sequence $(z^\k)_{\k\in\N}$ satisfies
\begin{equation}
  \forall \k\in\N: z^\k\in \ball \rho(z^*),\quad \sum_{\k=0}^\infty \vnorm{x^\k - x^\km} < \infty,\quad \FF(z^\k) \to \FF(z^*), \text{ as } \k\to\infty\,,
\end{equation}
$(z^\k)_{\k\in\N}$ converges to a point $\bar z=(\bar x,\bar x)\in \ball\sigma(z^*)$ such that $\FF(\bar z) \leq \FF(z^*)$.
If, additionally, \ref{ass:cont} is satisfied, then $0\in \partial \FF(\bar z)$ and $\FF(\bar z) = \FF(z^*)$.
\end{LEM}

Under Assumption~\ref{ass:local}, the following theorem establishes the local convergence result. Note that, thanks to Lemma~\ref{lem:local-growth}, a global minimizer automatically satisfies \ref{ass:local}. 
\begin{THM}[abstract local convergence] \label{thm:local-convergence-abstract}
Let $\map{\FF}{\R^{2N}}\eR$ be a proper lsc function which has the KL property at some local (or global) minimizer $z^*= (x^*,x^*)$ of $\FF$. Assume \ref{ass:local} holds at $z^*$.\\
Then, for any $r>0$, there exist $u\in\ivaloo{0,r}$ and $\mu >0$ such that the conditions 
\begin{equation} \label{eq:cond-local-min-thm}
  z^0 \in \ball u (z^*)\,,\qquad \FF(z^*) < \FF(z^0) < \FF(z^*) + \mu \,,
\end{equation}
imply that any sequence $\seq[\k\in\N]{z^\k}$ that starts at $z^0$ and satisfies \ref{ass:descent} and \ref{ass:error} has the finite length property \eqref{eq:thm:abstr-global-conv:finite-length-prop} and remains in $\ball r(z^*)$ and converges to some $\bar z \in \ball r(z^*)$, a critical point of $\FF$ with $\FF(\bar z) = \FF(z^*)$. For $r$ sufficiently small, $\bar z$ is a local minimizer of $\FF$.
\end{THM}
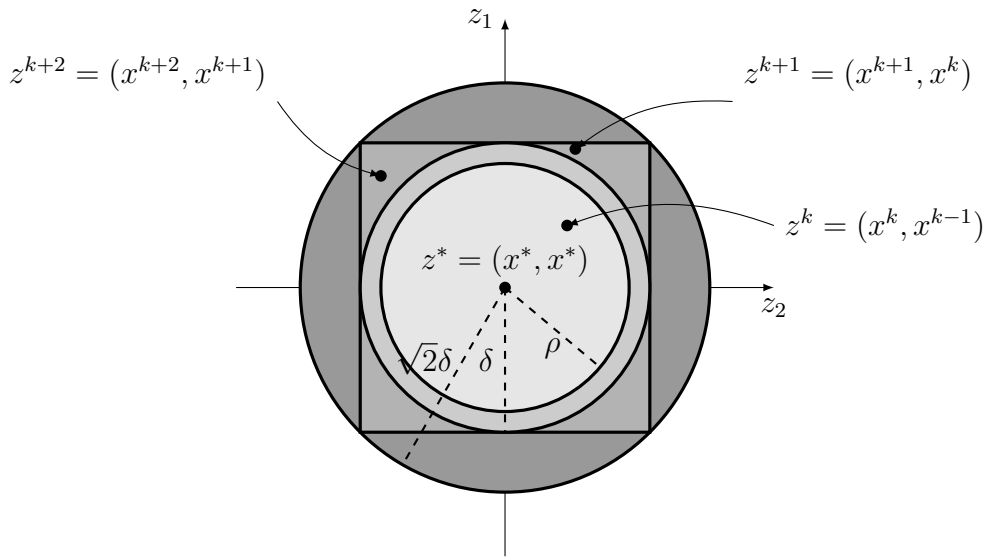
\begin{figure}
  \begin{center}
  \begin{tikzpicture}[scale=0.55]
    
    \draw[-latex] (-6.5,0) -- (6.5,0) node[below] {$z_2$};
    \draw[-latex] (0,-6.5) -- (0,6.5) node[left] {$z_1$};

    \draw[very thick,draw=black,fill=black!40] (0,0) circle [radius=4.949747cm];
    \draw[very thick,draw=black,fill=black!30] (-3.5,-3.5) rectangle (3.5,3.5);
    \draw[very thick,draw=black,fill=black!20] (0,0) circle [radius=3.5cm];
    \draw[very thick,draw=black,fill=black!10] (0,0) circle [radius=3cm];

    \fill (0,0) circle [radius=4pt] node[above] {$z^*=(x^*,x^*)$};
    \draw[thick,dashed] (0,0) -- (-90:3.5) node[pos=0.5,left] {$\delta$};
    \draw[thick,dashed] (0,0) -- (-120:4.949747) node[pos=0.4,left] {$\sqrt 2 \delta$};
    \draw[thick,dashed] (0,0) -- (-40:3) node[pos=0.5,below] {$\rho$};

    \fill (1.5,1.5) circle [radius=4pt]; 
    \draw[latex-] (1.5,1.5) to[bend left=20] (6.5,1.5) node[right] {$z^\k=(x^\k,x^\km)$};
    \fill (1.7,3.35) circle [radius=4pt]; 
    \draw[latex-] (1.7,3.35) to[bend left=20] (5.5,4.5) node[above right] {$z^\kp=(x^\kp,x^\k)$};
    \fill (-3,2.7) circle [radius=4pt]; 
    \draw[latex-] (-3,2.7) to[bend left=20] (-5.5,4.5) node[above left] {$z^\kpp=(x^\kpp,x^\kp)$};
  \end{tikzpicture}
  \end{center}
  \caption{\label{fig:proof-thm:local-convergence-abstract}An essential step of the proof of Theorem~\ref{thm:local-convergence-abstract} is to show: $z^\k\in \ball \rho(z^*)=\ball\rho(x^*,x^*)$ implies $x^\kpp, x^\kp \in \ball \delta (z^*_2)=\ball\delta(x^*)$ which restricts $z^\kp$ and $z^\kpp$ to the rectangle in the plot and thus to $\ball {\sqrt 2\delta}(z^*)$.}
\end{figure}
\begin{proof}
  Let $r>0$. Since $\FF$ satisfied the KL property at $z^*$ there exist $\eta_0 \in \ivaloc{0,+\infty}$, $\delta\in \ivaloo{0,r/\sqrt2}$ and a continuous concave function $\map \phi{\ivalco{0,\eta_0}}\R$ such that $\phi(0)=0$, $\phi$ is continuously differentiable and strictly increasing on $\ivaloo{0,\eta_0}$, and for all 
  \[
    z\in \ball  {\sqrt 2\delta} (z^*) \cap \ivalcc{\FF(z^*) < \FF(z) < \FF(z^*) + \eta_0}
  \]
  the KL inequality holds. As $z^*$ is a local minimizer, by choosing a smaller $\delta$ if necessary, one can assume that 
  \begin{equation} \label{eq:local-opt-cond}
    \FF(z) \geq \FF(z^*)\quad\text{ for all }\quad z\in \ball {\sqrt 2\delta}(z^*) \,.
  \end{equation}
  
  Let $0<\rho < \delta$ and $\nu>0$ be the parameters appearing in \ref{ass:local} with $\delta$ as in \eqref{eq:local-opt-cond}.  We want to verify the implication in \eqref{eq:mt-eqA} with $\sigma=\sqrt2\delta$. Let $\eta := \min(\eta_0,\nu)$ and $\k\in\N$. Assume $z^{0}, \ldots, z^{\k} \in \ball \rho (z^*)$, with $z^\k =: (z_1^\k,z_2^\k) = (x^\k,x^\km)\in \R^{\dimN\times2}$ and w.l.o.g. $\FF(z^*) < \FF(z^0) , \ldots, \FF(z^\k) < \FF(z^*) + \eta$ (note that if $\FF(z^\k) = \FF(z^*)$ the sequence is stationary ($x^k=x^{k+1}=x^{k+2}=\ldots$) by \ref{ass:descent} and the result follows directly). 

  See Figure~\ref{fig:proof-thm:local-convergence-abstract} for the idea of the following steps. First, note that $x^\k \in \ball \delta (z_2^*)$ as $z^\k \in \ball \delta (z^*)$. Suppose $z_2^\kpp=x^\kp \not\in \ball \delta (z_2^*)$. Then by \ref{ass:local} and \ref{ass:descent} we observe (use $(u+v)^2 \leq 2 (u^2 + v^2)$)
   \begin{multline*}
      \FF(z^\k) < \FF(z^\kpp) + \frac a4 \vnorm{x^\km - x^\kp}^2 \\
               \leq \FF(z^\k) - a\left( \vnorm{x^\kp - x^\k}^2 + \vnorm{x^\k - x^\km}^2\right) 
                               + \frac a4 \vnorm{x^\km - x^\kp}^2 
               \leq \FF(z^\k) \,,
  \end{multline*}
  which is a contradiction and therefore $z_2^\kpp\in \ball \delta (z_2^*)$. 
  
  Hence, due to the equivalence of norms in finite dimensions, $z^\kp = (x^\kp,x^\k) \in \ball {\sqrt2 \delta}(z^*)$. Thanks to \eqref{eq:local-opt-cond}, we have $\FF(z^\kp)\geq \FF(z^*)$. In order to verify \eqref{eq:mt-eqA}, we also need $\FF(z^\kpp) \geq \FF(z^*)$, which can be shown analogously, however we need to consider three iteration steps (that's the reason for the factor $\frac a4$ instead of $\frac a2$ on the right hand side of \ref{ass:local}). Assuming $z^\kppp_2 = x^\kpp \not\in \ball \delta (z_2^*)$ yields the following contradiction:
  \begin{multline*}
      \FF(z^\k) < \FF(z^\kppp) + \frac a4 \vnorm{x^\km - x^\kpp}^2 \\
               \leq \FF(z^\k) - a\left( \vnorm{x^\kpp - x^\kp}^2+\vnorm{x^\kp - x^\k}^2 + \vnorm{x^\k - x^\km}^2\right)  + \frac a4 \vnorm{x^\km - x^\kpp}^2  \\
               \leq \FF(z^\k) - a\left( \vnorm{x^\kpp - x^\kp}^2+\vnorm{x^\kp - x^\k}^2 + \vnorm{x^\k - x^\km}^2\right)  \\
               + \frac a4 \left( 2\vnorm{x^\kpp - x^\kp}^2+4\vnorm{x^\kp - x^\k}^2 + 4\vnorm{x^\k - x^\km}^2\right) 
               \leq \FF(z^\k) \,.
  \end{multline*}
  Therefore, $\FF(z^\kp),\FF(z^\kpp) \geq \FF(z^*)$ holds, which is exactly property \eqref{eq:mt-eqA} with $\sigma = \sqrt 2 \delta$. 
  
  Now, choose $u,\mu >0$ in \eqref{eq:cond-local-min-thm} such that
  \[
      \mu < \eta\,,\ u < \frac \rho 6\,,\ \sqrt{\frac \mu a} + \frac ba \phi(\mu) < \frac{\rho}3 \,.
  \]
  If $z^0$ satisfies \eqref{eq:cond-local-min-thm}, we have
  \[
    \vnorm{x^*-x^0} + \sqrt{\frac{\FF(z^0) - \FF(z^*)}a} + \frac ba \phi(\FF(z^0) - \FF(z^*)) < \frac \rho 2 \,,
  \]
  which is \eqref{eq:mt-eqD} with $\mu$ in place of $\eta$. Using Lemma~\ref{lem:main-theorem-convergence} we conclude that the sequence has the finite length property, remains in $\ball \rho(z^*)$, converges to $\bar z \in \ball \sigma(z^*)$, $\FF(z^\k) \to \FF(z^*)$ and $\FF(\bar z) \leq \FF(z^*)$, which is only allowed for $\FF(\bar z) = \FF(z^*)$. Therefore, the sequence also has property \ref{ass:cont}, and thus, $\bar z$ is a critical point of $\FF$. The property in \eqref{eq:local-opt-cond} shows that $\bar z$ is a local minimizer for sufficiently small $r$.
\end{proof}
\begin{REM}
The assumption in \ref{ass:local} and Lemma~\ref{lem:local-growth} only restrict the behavior of the function along the second block of coordinates of $z=(z_1,z_2)\in\R^{2N}$. This makes sense, because, for sequences that we consider, the first and second block depend on each other.
\end{REM}
\begin{REM}
Unlike Theorem~\ref{thm:abstr-global-conv}, the local convergence theorem (Theorem~\ref{thm:local-convergence-abstract}) does not require assumption \ref{ass:cont} explicitly. If Theorem~\ref{thm:abstr-global-conv} assumes the KL property at some $z^*$ (not the cluster point $\tilde z$ of \ref{ass:cont}), convergence to a point $\bar z$ in a neighborhood of $z^*$ with $\FF(\bar z) \leq \FF(z^*)$ can be shown. However, $\FF(\bar z) < \FF(z^*)$ might happen, which disproves $\FF$-attentive convergence of $z^\k\to\bar z$, thus $\bar z$ would not be a critical point. Assuming $\tilde z= z^*$ by \ref{ass:cont} assures the $\FF$-attentive convergence, and thus $\bar z$ is a critical point. Because of the local minimality of $z^*$ in Theorem~\ref{thm:local-convergence-abstract} $\FF(\bar z) < \FF(z^*)$ cannot occur, and therefore \ref{ass:cont} is implied.
\end{REM}
Before deriving the convergence rates, we apply Theorem~\ref{thm:local-convergence-abstract} and Lemma~\ref{lem:local-growth} to show a useful example of a feasibility problem.
\begin{EX}[semi-algebraic feasibility problem] \label{ex:local-convergence-feasibility}
  Let $S_1,\ldots,S_M\subset\R^\dimN$ be semi-algebraic sets such that $\bigcap_{i=1}^M S_i \neq \emptyset$  and let $\map \F{\R^\dimN}{\eR}$ be given by $\F(x) = \frac 12 \sum_{i=1}^M \dist(x,S_i)^2$. For a constant $c\geq0$, we consider the function $\FF(z) = \FF(z_1,z_2) = \F(z_1) + c \vnorm{z_1-z_2}^2$.  Suppose $z^*= (x^*,x^*)$ is a global minimizer of $\FF$, i.e., $x^* \in \bigcap_{i=1}^M S_i$. Then, for $z^0=(x^0,x^{-1})$ sufficiently close to $z^*$, any algorithm that satisfies \ref{ass:descent} and \ref{ass:error} and starts at $z^0$ generates a sequence that remains in a neighborhood of $z^*$,  has the finite length property, and converges to a point $\bar z =(\bar x, \bar x)$ with $\bar x \in \bigcap_{i=1}^M S_i$.
\end{EX}

Finally, we complement our local convergence result by the convergence rate estimates from \cite{JM16,LP17}. Assuming the objective function is semi-algebraic, in \cite[Theorems 2 and 4]{JM16} which build on \cite[Theorem 3.4]{FGP14}, a list of qualitative convergence rate estimates in terms of the KL-exponent is proved. For estimations on the KL-exponent, the interested reader is referred to \cite{LP17,BNPS16,LMP15,LMNP16}, which include estimations of the KL-exponent for convex polynomials, functions that can be expressed as the maximum of finitely many polynomials, functions that can be expressed as supremum of a collection of polynomials over a semi-algebraic compact set under suitable regularity assumptions, and relations to the Luo--Tseng error bound.
\begin{THM}[convergence rates] \label{thm:abstr-conv-rates}
Let $\seq[\k\in\N]{z^\k} = \seq[\k\in\N]{x^\k,x^{\km}}$ be a sequence that satisfies \ref{ass:descent}, \ref{ass:error}, and \ref{ass:cont} for a proper lsc function $\map{\FF}{\R^{2N}}\eR$ which has the KL property at the critical point $\tilde z=z^*$ specified in \ref{ass:cont}. Let $\theta$ be the KL-exponent of $\FF$. 
\begin{enumerate}
  \item If $\theta=1$, then $z^k$ converges to $z^*$ in a finite number of iterations.
  \item If $\frac 12 \leq \theta <1$, then $\FF(z^\k) \to \FF(z^*)$ and $x^\k\to x^*$ linearly.
  \item If $0<\theta < \frac 12$, then $\FF(z^\k)-\FF(z^*) \in O(\k^{\frac{1}{2\theta-1}})$ and $\vnorm{x^\k-x^*} \in O(\k^{\frac{\theta}{2\theta-1}})$.
\end{enumerate}
\end{THM}
\begin{proof}

  Using Theorem~\ref{thm:abstr-global-conv} the sequence $\seq[\k\in\N]{z^\k}$ converges to $z^*$ and $\FF(z^\k) \to \FF(z^*)$ as $\k\to\infty$. W.l.o.g. we can assume that $\FF(z^\k) > \FF(z^*)$ for all $\k\in\N$. By convergence of $\seq[\k\in\N]{z^\k}$ and \ref{ass:descent}, there exists $\k_0$ such that the KL-inequality \eqref{eq:KL-ineq} with $f=\FF$ holds for all $\k\geq \k_0$. Let $U$, $\phi$, $\eta$ be the objects appearing in Definition~\ref{def:KL-property}. Now, using $(u+v)^2 \leq 2(u^2+v^2)$ for $u,v\in\R$ to bound the terms on the right hand side of \ref{ass:error} and substituting \ref{ass:descent} into the resulting terms, the squared KL-inequality \eqref{eq:KL-ineq} at index $\k$ yields 
  \[
    \frac{b^2}{2a} \big( \phi^\prime(\FF(z^\k) - \FF(z^*)) \big)^2 \big(\FF(z^\km) - \FF(z^\kp)\big) \geq 1 \,.
  \]
  As $\phi^\prime(s)=cs^{\theta-1}$ is non-increasing for $\theta\in\ivalcc{0,1}$, we have $\phi^\prime(\FF(z^\k) - \FF(z^*))\leq \phi^\prime(\FF(z^\kp) - \FF(z^*))$.  The remainder of the proof is identical to \cite{JM16} starting from \cite[Inequality~(7)]{JM16}, which yields the rates for $\seq[\k\in\N]{\FF(z^\k)}.$ \\

In the following, we prove the rates for $\seq[\k\in\N]{x^\k}$. We make use of an intermediate result from the proof of \cite[Lemma 3.5]{OCBP14} (cf. Lemma~\ref{lem:main-theorem-convergence}). The starting point is \cite[Inequality (6)]{OCBP14} restricted to terms with index $k\geq K$ for some $K\in\N$: 
\[
  \sum_{\k\geq K} \vnorm{x^\k - x^\km} \leq \frac 12 \vnorm{x^K-x^{K-1}} + \frac ba \phi(\FF(z^K)-\FF(z^*)) \,.
\]

The triangle inequality shows that the left hand side is an upper bound for $\vnorm{x^K-x^*}$. Using \ref{ass:descent} to bound the right side of the preceding inequality yields:

\[
  \vnorm{x^K-x^*} \leq \sum_{\k\geq K} \vnorm{x^\k - x^\km} \leq c^{\prime\prime} \left( \phi(\FF(z^K)-\FF(z^*)) + \sqrt{\FF(z^K)-\FF(z^*)} \right) 
\]
for some constant $c^{\prime\prime}>0$. If the KL-exponent is $\theta\in \ivalco{\frac 12,1}$, for $\FF(z^K)-\FF(z^*)<1$,  the second term upper-bounds the first one, and $\FF(z^K)\to \FF(z^*)$ is linear. For $\theta\in \ivaloo{0,\frac 12}$ convergence is dominated by the first term, hence $\vnorm{x^K-x^*} \in O(\phi(\FF(z^K)-\FF(z^*)))$, which concludes the proof.
\end{proof}


\section{Local and Global Convergence of iPiano} \label{sec:review-of-iPiano}

In this section, we briefly review the method iPiano and verify that the abstract convergence results from Section~\ref{sec:abstract-convergence} hold for this algorithm.

iPiano applies to structured non-smooth and non-convex optimization problems with a proper lower semi-continuous (lsc) extended-valued function $\map{h}{\R^N}{\eR}$, $N\geq 1$:
\begin{equation}\label{eq:ipiano-class}
\min_{x \in \R^N} \; h(x)\,,\qquad h(x)= f(x) + g(x) 
\end{equation}
that satisfies the following assumption.
\begin{ASS} \label{ass:ipiano}
For $U\subset\R^N$, the following properties hold:
\begin{itemize}
  \item The function $\map f{U}{\R}$ is assumed to be $C^1$-smooth (possibly non-convex) with $L$-Lipschitz continuous gradient on $\dom g\cap U$, $L>0$. 
  \item The function $\map g {U}{\eR}$ is proper, lsc, possibly non-smooth and non-convex, simple and prox-bounded. 
  \item The function $h$ restricted to $U$ is bounded from below by some value $\underline h > -\infty$ and coercive, i.e., $\vnorm{x}\to \infty$ with $x\in U$ implies that $h(x)\to \infty$.
\end{itemize}
\end{ASS}
\begin{REM}
As we will use Assumption~\ref{ass:ipiano} either with $U=\R^N$ or $U=\ball{r^\prime}(x^*)$ for some $r^\prime>0$, the coercivity assumption reduces either to the usual definition ($U=\R^N$) or is empty (since $\ball{r^\prime}(x^*)$ is bounded). The coercivity property could be replaced by the assumption that the sequence that is generated by the algorithm is bounded.
\end{REM}
\begin{REM}
Simple refers to the fact that the associated proximal map can be solved efficiently for the global optimum. 
\end{REM}
iPiano is outlined in Algorithm~\ref{alg:ipiano-gen}. For $g=0$, iPiano coincides with the Heavy-ball method (inertial gradient descent method or gradient descent with momentum). 

In \cite{Ochs15}, functions $g$ that are semi-convex received special attention. The resulting step size restrictions for semi-convex functions $g$ are similar to those of convex functions. A function is said to be semi-convex with modulus $m\in \R$, if $m$ is the largest value such that $g(x) - \frac m2 \vnorm{x}^2$ is convex. For convex functions, $m=0$ holds, and for strongly convex functions $m>0$. We assume $m< L$.
  According to \cite[Theorem 10.33]{Rock98}, saying a function $g$ is (locally) semi-convex on an open set $V\subset\dom g$ is the same as saying $g$ is lower-$\mathcal C^2$ on $V$. Nevertheless, the function $g$ does not need to be semi-convex. This property is just used to improve the bounds on the step size parameters. 
\begin{REM}
  For simplicity, we describe the constant step size version of iPiano. However, all results in this paper are also valid for the backtracking line-search version of iPiano.
\end{REM}
\begin{figure}[t]
\centering
\setlength{\fboxsep}{5pt}%
\setlength{\fboxrule}{1.5pt}%
\fbox{
\begin{minipage}{0.95\linewidth}
\begin{ALG} iPiano \ \label{alg:ipiano-gen}
\begin{itemize}
\item \textbf{Optimization problem:} 
\[
  \text{\eqref{eq:ipiano-class} with Assumption~\ref{ass:ipiano} for}\ 
 \begin{cases} 
        U=\R^N \\
        U=\ball {r^\prime} (x^*)\ \text{ for a local minimizer $x^*$ and $r^\prime>0$.}
     \end{cases}
\]
\item \textbf{Initialization}: Choose a starting point $x^0 \in \dom h\cap U$ and set $x^{-1}= x^0$.
\item \textbf{Iterations $(\k \geq 0)$}:  Update:
  \begin{equation}\label{eq:ipiano-gen-up}
    \begin{split}
    y^{\k} =&\ x^{\k} + \beta(x^{\k} - x^{\km})  \\
    x^{\kp} \in&\ \arg\min_{x\in\R^N}\ g(x) + \scal{\nabla f(x^\k)}{x-x^\k} + \frac{1}{2\alpha} \vnorm{x - y^\k}^2  \,.
    \end{split}
  \end{equation}
\item \textbf{Parameter setting}: See Table~\ref{tab:iPiano-param-conv}.
\end{itemize}
\end{ALG}
\end{minipage}
}
\end{figure}

\begin{table}
\begin{center}
\begin{tabular}{|c||c|c|c|c|}
\hline
  Method & $f$ & $g$ & $\alpha$ & $\beta$ \\
\hline\hline
  Gradient Descent & 
  $f\in \mathcal C^{1+}$ & $g\equiv 0$ & $\alpha \in \ivaloo{0,\frac 2L}$ & $\beta=0$  \\
\hline 
  Heavy-ball method & 
  $f\in \mathcal C^{1+}$ & $g\equiv 0$ & $\alpha \in \ivaloo{0,\frac {2(1-\beta)}L}$ & $\beta\in \ivalco{0,1}$  \\
\hline
  PPA  & 
  $f\equiv 0$ & $g$ convex & $\alpha > 0$ & $\beta=0$  \\
\hline
  FBS & 
  $f\in \mathcal C^{1+}$ & $g$ convex & $\alpha \in \ivaloo{0,\frac 2L}$ & $\beta=0$  \\
\hline 
  FBS (non-convex) & 
  $f\in \mathcal C^{1+}$ & $g$ non-convex & $\alpha \in \ivaloo{0,\frac 1L}$ & $\beta=0$  \\
\hline 
  iPiano & 
  $f\in \mathcal C^{1+}$ & $g$ convex & $\alpha \in \ivaloo{0,\frac {2(1-\beta)}L}$ & $\beta\in \ivalco{0,1}$  \\
\hline 
  iPiano & 
  $f\in \mathcal C^{1+}$ & $g$ non-convex & $\alpha \in \ivaloo{0,\frac {(1-2\beta)}L}$ & $\beta\in \ivalco{0,\frac 12}$  \\
\hline 
  iPiano & 
  $f\in \mathcal C^{1+}$ & $g$ $m$-semi-convex & $\alpha \in \ivaloo{0,\frac {2(1-\beta)}{L-m}}$ & $\beta\in \ivalco{0,1}$  \\
  \hline
\end{tabular}
\end{center}
\caption{\label{tab:iPiano-param-conv}Convergence of iPiano as stated in Corollaries~\ref{cor:convergence-whole-seq}, \ref{cor:local-convergence-iPiano} and~\ref{cor:ipiano-conv-rates} is guaranteed for the parameter settings listed in this table (for $g$ convex, see \cite[Algorithm 2]{OCBP14}, otherwise see \cite[Algorithm 3]{Ochs15}). Note that for local convergence, also the required properties of $f$ and $g$ are required to hold only locally. iPiano has several well-known special cases, such as the gradient descent method, Heavy-ball method, proximal point algorithm (PPA), and forward--backward splitting (FBS). $\mathcal C^{1+}$ denotes the class of functions whose gradient is strictly continuous (Lipschitz continuous).}
\end{table}

The following convergence results hold for iPiano.
\begin{COR}[global convergence of iPiano {\cite[Theorem 6.6]{Ochs15}}] \label{cor:convergence-whole-seq}
Let $\seq[\k\in\N]{x^\k}$ be generated by Algorithm~\ref{alg:ipiano-gen} with $U=\R^N$. Then, the sequence $\seq[\k\in\N]{z^\k}$ with $z^\k=(x^\k, x^\km)$ satisfies \ref{ass:descent}, \ref{ass:error}, \ref{ass:cont} for the function (for some $\ipianodelta >0$)
\begin{equation} \label{eq:cor:convergence-whole-seq:H-fun}
  \map{H_\ipianodelta}{\R^{2N}}{\R\cup\{\infty\}}\,,\quad (x,y)\mapsto h(x) + \ipianodelta \vnorm{x-y}^2 \,.
\end{equation}

Moreover, if $H_\ipianodelta(x,y)$ has the \KL property at a cluster point $z^*=(x^*,x^*)$, then the sequence $\seq[\k\in\N]{x^\k}$ has the finite length property, $x^\k\to x^*$ as $\k\to\infty$, and $z^*$ is a critical point of $H_\ipianodelta$, hence $x^*$ is a critical point of $h$.
\end{COR}
\begin{COR}[local convergence of iPiano] \label{cor:local-convergence-iPiano} 
Let $\seq[\k\in\N]{x^\k}$ be generated by Algorithm~\ref{alg:ipiano-gen} with $U=\ball{r^\prime}(x^*)$ for some $r^\prime>0$, where $x^*$ is a local (or global) minimizer of $h$. Then $z^*=(x^*,x^*)$ is a local (or global) minimizer of $H_\ipianodelta$ (defined in \eqref{eq:cor:convergence-whole-seq:H-fun}). Suppose \ref{ass:local} holds at $z^*$ and $H_\ipianodelta$ has the KL property at $z^*$. 

Then, for any $r>0$ (in particular for $r=r^\prime$), there exist $u\in\ivaloo{0,r}$ and $\mu >0$ such that the conditions 
  \[
  x^0 \in \ball u (x^*)\,,\qquad h(x^*) < h(x^0) < h(x^*) + \mu \,,
  \]
imply that the sequence $\seq[\k\in\N]{x^\k}$ has the finite length property and remains in $\ball r(x^*)$ and converges to some $\bar x \in \ball r(x^*)$, a critical point of $h$ with $h(\bar x) = h(x^*)$. For $r$ sufficiently small, $\bar z$ is a local minimizer of $h$.
\end{COR}
\begin{proof}
  Corollary~\ref{cor:convergence-whole-seq} shows that Algorithm~\ref{alg:ipiano-gen} generates a sequence that satisfies \ref{ass:descent}, \ref{ass:error}, \ref{ass:cont} with $H_\ipianodelta$. Therefore, obviously, Theorem~\ref{thm:local-convergence-abstract} can be applied.
\end{proof}
\begin{COR}[convergence rates for iPiano] \label{cor:ipiano-conv-rates}
Let $\seq[\k\in\N]{x^\k}$ be generated by Algorithm~\ref{alg:ipiano-gen} and set $z^\k:=(x^\k,x^\km)$. If $H_\ipianodelta$, defined in \eqref{eq:cor:convergence-whole-seq:H-fun}, has the KL property at $z^*=(x^*,x^*)$ specified in \ref{ass:cont} with KL-exponent $\theta$, then the following rates of convergence hold for some $C>0$:
\begin{enumerate}
  \item If $\theta=1$, then $x^k$ converges to $x^*$ in a finite number of iterations.
  \item If $\frac 12 \leq \theta <1$, then $h(x^\k) \to h(x^*)$ and $x^\k\to x^*$ linearly.
  \item If $0<\theta < \frac 12$, then $h(x^\k)-h(x^*) \in C(\k^{\frac{1}{2\theta-1}})$ and $\vnorm{x^\k-x^*} \in O(\k^{\frac{\theta}{2\theta-1}})$.
\end{enumerate}
\end{COR}
\begin{proof}
  Corollary~\ref{cor:convergence-whole-seq} shows that Algorithm~\ref{alg:ipiano-gen} generates a sequence that satisfies \ref{ass:descent}, \ref{ass:error}, \ref{ass:cont} for $H_\ipianodelta$. Therefore, the statement follows from Theorem~\ref{thm:abstr-conv-rates} and the facts that $H_\ipianodelta(x^*,x^*)=h(x^*)$ and $h(x^\k) \leq H_\ipianodelta(x^\k,x^\km)$.
\end{proof}
\begin{REM} \label{rem:KL-for-iPiano-Lyapunov}
  In \cite[Theorem 3.6]{LP17}, Li and Pong show that, if $h$ has the KL-exponent $\theta\in \ivaloc{0,\frac 12}$ at $x^*$, then $H_\ipianodelta$ has the same KL-exponent at $z^*=(x^*,x^*)$.
\end{REM}

\section{Inertial Averaged/Alternating Minimization} \label{sec:alt-avg-prox-min}

In this section, we transfer the convergence result developed for iPiano in Section~\ref{sec:review-of-iPiano} to various non-convex settings (Section~\ref{subsec:Heavy-ball-on-Moreau}, \ref{subsec:Heavy-ball-on-sum-Moreaus} and \ref{subsec:ipiano-on-Moreau}). This yields inertial algorithms for non-convex problems that are known from the convex setting as averaged or alternating proximal minimization (or projection) methods. Key for the generalization to the non-convex and inertial setting are an explicit formula for the gradient of the Moreau envelope of a prox-regular function (Proposition~\ref{prop:lipschitz-Moreau-envelope}), which is well-known for convex functions (Proposition~\ref{prop:lipschitz-Moreau-envelope-convex}), and the local convergence result in Theorem~\ref{thm:local-convergence-abstract}. For completeness, we state the formula in the convex setting, before we devote ourselves to the prox-regular setting. 
\begin{PROP}[{\cite[Proposition 12.29]{BC11}}]\label{prop:lipschitz-Moreau-envelope-convex}
  Let $\map{f}{\R^N}{\eR}$ be a proper lower semi-continuous (lsc) convex function and $\lambda>0$. Then $\menv \lambda f$ is continuously differentiable and its gradient 
\begin{equation} \label{eq:Moreau-grad}
  \nabla \menv\lambda f(x) =\frac{1}{\lambda} (x - \prox\lambda f(x)) \,,
\end{equation}
is $\lambda^{-1}$-Lipschitz continuous.
\end{PROP}
\begin{PROP}\label{prop:lipschitz-Moreau-envelope} 
  Suppose that $\map f{\R^\dimN}{\eR}$ is prox-regular at $\bar x$ for $\bar v =0$, and that $f$ is prox-bounded. Then for all $\lambda >0$ sufficiently small there is a neighborhood of $\bar x$ on which 
  \begin{enumerate}
  \item\label{prop:lipschitz-Moreau-envelope-A} $\prox \lambda f$ is monotone, single-valued and Lipschitz continuous and $\prox\lambda f (\bar x) = \bar x$.
  \item\label{prop:lipschitz-Moreau-envelope-B} $\menv \lambda f$ is differentiable with $\nabla (\menv \lambda f )(\bar x) = 0$, in fact $\nabla (\menv \lambda f )$ is strictly continuous with
  \begin{equation} \label{eq:prox-reg-Moreau-formula}
      \nabla \menv \lambda f  = \lambda^{-1} (I - \prox \lambda f) = (\lambda I + \floc^{-1})^{-1}
  \end{equation}
  for an $f$-attentive localization $\floc$ of $\partial f$ at $(\bar x,0)$, where $I$ denotes the identity mapping. Indeed, this localization can be chosen so that the set $U_\lambda := \rge(I+\lambda\floc)$ serves for all $\lambda >0$ sufficiently small as a neighborhood of $\bar x$ on which these properties hold. 
  \item\label{prop:lipschitz-Moreau-envelope-C} There is a neighborhood of $\bar x$ on which for small enough $\lambda$ the local Lipschitz constant of $\nabla \menv \lambda f$ is $\lambda^{-1}$. 
  If $\lambda_0$ is the modulus of prox-regularity at $\bar x$, then $\lambda \in \ivaloo{0, \lambda_0/2}$ is a sufficient condition.
  \item\label{prop:lipschitz-Moreau-envelope-D} Any point $\tilde x\in U_\lambda$ with $\nabla \menv \lambda f (\tilde x) = 0$ is a fixed point of $\prox\lambda f$ and a critical point of $f$.
  \end{enumerate}
\end{PROP}
\begin{proof}
  While Item~\ref{prop:lipschitz-Moreau-envelope-A} and~\ref{prop:lipschitz-Moreau-envelope-B} are proved in \cite[Proposition 13.37]{Rock98}, Item~\ref{prop:lipschitz-Moreau-envelope-C} (estimation of the local Lipschitz constant) and~\ref{prop:lipschitz-Moreau-envelope-D} are not explicitly verified. In order to prove Items~\ref{prop:lipschitz-Moreau-envelope-C} and~\ref{prop:lipschitz-Moreau-envelope-D}, we develop the basic objects that are required in the same way as \cite[Proposition 13.37]{Rock98}. Thus, the first part of the proof coincides with \cite[Proposition 13.37]{Rock98}.\\
  
  Without loss of generality, we can take $\bar x =0$. As $f$ is prox-bounded the condition for prox-regularity may be taken to be global, cf. \cite[Proposition 8.46(f)]{Rock98}, i.e., there exists $\eps >0$ and $\lambda_0>0$ such that
  \begin{gather} 
\label{eq:subgrad-ineq-proof}
    f(x^\prime) > f(x) + \scal{v}{x^\prime - x} - \frac {1}{2\lambda_0} \vnorm{x^\prime - x}^2 \quad \forall x^\prime\neq x \\
\label{eq:localization}
    \text{when}\ v\in \partial f(x),\ \vnorm{v} < \eps,\ \vnorm{x}<\eps,\ f(x) < f(0) + \eps \,.
  \end{gather}
  Let $\smap{\floc}{\R^N}{\R^N}$ be the $f$-attentive localization of $\partial f$ specified in \eqref{eq:localization}, i.e. the set-valued mapping defined by $\Graph\floc = \set{(x,v)\setsep v\in \partial f(x),\,\vnorm{v} < \eps,\,\vnorm{x}<\eps,\, f(x) < f(0) + \eps}$. Inequality \eqref{eq:subgrad-ineq-proof} is valid for any $\lambda \in \ivaloo{0,\lambda_0}$. Setting $u=x+\lambda v$ the subgradient inequality \eqref{eq:subgrad-ineq-proof} (with $\lambda$ instead of $\lambda_0$) implies
  \[
    f(x^\prime) + \frac{1}{2\lambda} \vnorm{x^\prime - u}^2 > f(x) + \frac{1}{2\lambda}\vnorm{x-u}^2 \,.
  \]
  Therefore, $\prox\lambda f(x+\lambda v)=\set{x}$ when $v\in \floc(x)$. In general, for any $u$ sufficiently close to $0$, thanks to Fermat's rule on the minimization problem of $\prox\lambda f(u)$, we have for any $x\in \prox\lambda f(u)$ that $v=(u-x)/\lambda \in \floc (x)$ holds. Thus, $U_\lambda=\rge(I+\lambda\floc)$ is a neighborhood of $0$ on which $\prox\lambda f$ is single-valued and coincides with $(I+\lambda \floc)^{-1}$.

  \ref{prop:lipschitz-Moreau-envelope-C} %
  Now, let $u=x+\lambda v$ and $u^\prime=x^\prime + \lambda v^\prime$ be any two elements in $U_\lambda$ such that $x=\prox\lambda f(u)$ and $x^\prime = \prox\lambda f(u^\prime)$. Then $(x,v)$ and $(x^\prime,v^\prime)$ belong to $\Graph\floc$. Therefore, we can add two copies of \eqref{eq:subgrad-ineq-proof} where in the second copy the roles of $x$ and $x^\prime$ are swapped. This sum yields for any $\lambda_1\in\ivaloo{0,\lambda_0}$ instead of $\lambda_0$ in \eqref{eq:subgrad-ineq-proof}:
  \begin{equation}\label{eq:hypromonotone}
    0 \geq \scal{v-v^\prime}{x^\prime - x} - \frac 1{\lambda_1} \vnorm{x^\prime - x}^2 \,.
  \end{equation}
  In this inequality, we substitute $v$ with $(u-x)/\lambda$ and $v^\prime$ with $(u^\prime-x^\prime)/\lambda$ which yields
  \[
    0 \leq \frac 1{\lambda_1} \vnorm{x^\prime - x}^2 + \frac 1\lambda \scal{(u^\prime-x^\prime)-(u-x)}{x^\prime - x} =  \frac 1\lambda \scal{u^\prime-u}{x^\prime - x}  +\left(\frac 1{\lambda_1} - \frac 1{\lambda}\right)  \vnorm{x^\prime - x}^2 
  \]
  or, equivalent to that $\scal{u^\prime-u}{x^\prime - x} \geq (1-\tfrac{\lambda}{\lambda_1}) \vnorm{x^\prime - x}^2$.
  
  This expression helps to estimate the local Lipschitz constant of the gradient of the Moreau envelope. Using the closed form description of $\nabla \menv \lambda f$ on $U_\lambda$, we verify the $\lambda^{-1}$-Lipschitz continuity of $\nabla \menv\lambda f$ as follows: 
  \[
    \begin{split}
    \lambda^2\vnorm{\nabla \menv\lambda f(u) - \nabla \menv \lambda f(u^\prime)}^2 - \vnorm{u-u^\prime}^2 
    =&\  \vnorm{ (u- u^\prime) - (\prox \lambda f (u) - \prox \lambda f(u^\prime))}^2 - \vnorm{u-u^\prime}^2 \\
    =&\ \vnorm{x - x^\prime}^2 - 2 \scal{u- u^\prime}{x-x^\prime}  \\
    \leq&\ (2\tfrac{\lambda}{\lambda_1}-1) \vnorm{x-x^\prime}^2 \leq 0 
    \end{split}
  \]
  when  $\lambda \leq \frac 12 \lambda_1$. \\

  \ref{prop:lipschitz-Moreau-envelope-D} %
  Now, let $\tilde x\in U_\lambda$ be a point for which $\nabla \menv \lambda f (\tilde x) = 0$ holds. Then, according to \eqref{eq:prox-reg-Moreau-formula}, we have $\tilde x = \prox\lambda f(\tilde x)$ or $\tilde x = (I + \lambda T)^{-1}(\tilde x)$ for the localization selected above. Inverting the mapping shows that $\tilde x \in \tilde x + \lambda T(\tilde x)$, which implies that $0\in T(\tilde x)$, thus $0\in \partial f(\tilde x)$.
\end{proof}
\begin{REM}
 The proof of Item (iii) of Proposition~\ref{prop:lipschitz-Moreau-envelope} is motivated by a similar derivation for distance functions and projection operators in \cite{LLM08}. See \cite{JTZ14}, for a recent analysis of the differential properties of the Moreau envelope in the infinite dimensional setting.
\end{REM}

\subsection{Heavy-ball Method on the Moreau Envelope}  \label{subsec:Heavy-ball-on-Moreau}

\begin{PROP}[inertial proximal minimization method] \label{prop:inertial-prox-min}
  Suppose $\map{f}{\R^N}{\eR}$ is prox-regular at $x^*$ for $v^*=0$ with modulus $\lambda_0>0$ and prox-bounded with threshold $\lambda_f>0$. Let $0<\lambda< \min(\lambda_f, \lambda_0/2)$, $\beta\in\ivalco{0,1}$, and $\alpha \in \ivaloo{0,2(1-\beta)\lambda}$. Suppose that $h = \menv\lambda f$ has a local minimizer $x^*$ and $H_\ipianodelta$, defined in \eqref{eq:cor:convergence-whole-seq:H-fun}, satisfies \ref{ass:local} and the KL property at $(x^*,x^*)$.
  
  Let $x^0=x^{-1}$ with $x^0\in \R^N$ and let the sequence $\seq[\k\in\N]{x^\k}$ be generated by the following update rule
  \[
     x^\kp \in (1-\alpha\lambda^{-1}) x^\k + \alpha\lambda^{-1} \prox \lambda f (x^\k) + \beta (x^\k - x^\km) \,.
  \]
  If $x_0$ is sufficiently close to $x^*$, then sequence $\seq[\k\in\N]{x^\k}$ 
  \begin{itemize}
    \item is uniquely determined,
    \item has the finite length property, 
    \item remains in a neighborhood of $x^*$, 
    \item and converges to a critical point $\tilde x$ of $f$ with $f(\tilde x) = f(x^*)$.
  \end{itemize}
  If $f$ is proper, lsc, convex, and $\lambda >0$, $\beta\in\ivalco{0,1}$, and $\alpha \in \ivaloo{0,2(1-\beta)\lambda}$, then the sequence has finite length and converges to a global minimizer $\tilde x$ of $f$ for any $x^0\in \R^N$.
\end{PROP}
\begin{proof}
  The statement is an application of the results for the Heavy-ball method (i.e. \eqref{eq:ipiano-class} with $g\equiv0$) to the Moreau envelope $\menv \lambda f$ of the function $f$. Note that $H_\ipianodelta$ inherits the KL-property from $h$ (see Remark~\ref{rem:KL-for-iPiano-Lyapunov}).

  Since $f$ is prox-bounded with threshold $\lambda_f$, the function is bounded from below and coercive for $\lambda < \lambda_f$. As $\lambda < \lambda_0/2$, Proposition~\ref{prop:lipschitz-Moreau-envelope} can be used to conclude that there exists a neighborhood $U_\lambda$ of $x^*$ such that $\menv\lambda f$ is differentiable on $U_\lambda$ and $\nabla \menv\lambda f$ is $\lambda^{-1}$-Lipschitz continuous. 

  There exists a neighborhood $U\subset U_\lambda$ of $x^*$ which contains $x_0$ and Corollary~\ref{cor:local-convergence-iPiano} can be applied. Therefore, the Heavy-ball method (Algorithm~\ref{alg:ipiano-gen} with $g\equiv 0$) with $0<\alpha < 2(1-\beta)\lambda$ and $\beta \in \ivalco{0,1}$ generates a sequence $\seq[\k\in\N]{x^k}$ that lies in $U$. Using the formula in \eqref{eq:prox-reg-Moreau-formula}, the update step of the Heavy-ball method applied to $\menv \lambda f$ reads as follows:
  \[
    \begin{split}
    x^\kp =&\  x^\k - \alpha \nabla \menv \lambda f (x^\k) + \beta(x^\k - x^\km)  \\
          =&\  x^\k - \alpha \lambda^{-1}(x^\k - \prox \lambda f (x^\k)) + \beta(x^\k - x^\km) \\
          =&\  (1- \alpha \lambda^{-1}) x^\k + \alpha \lambda^{-1} \prox \lambda f (x^\k) + \beta (x^\k - x^\km)\,.
    \end{split}
  \]
  By Proposition~\ref{prop:lipschitz-Moreau-envelope}\ref{prop:lipschitz-Moreau-envelope-A} $\prox \lambda f$ is single-valued and by Proposition~\ref{prop:lipschitz-Moreau-envelope}\ref{prop:lipschitz-Moreau-envelope-D} $0\in \partial f(\tilde x)$. The remaining statements follow follow from Corollary~\ref{cor:local-convergence-iPiano}.

  The statement about convex functions $f$ follows analogously by using Proposition~\ref{prop:lipschitz-Moreau-envelope-convex} instead of Proposition~\ref{prop:lipschitz-Moreau-envelope} and Corollary~\ref{cor:convergence-whole-seq} instead of Corollary~\ref{cor:local-convergence-iPiano}.
\end{proof}
\begin{REM}
  Corollary~\ref{cor:ipiano-conv-rates} provides a list of convergence rates for the method in Proposition~\ref{prop:inertial-prox-min}.
\end{REM}
\begin{REM}
  The question whether $h = \menv\lambda f$ has the KL property if $f$ has the KL property has been analyzed for convex functions in \cite{LP17}. For non-convex functions, this is a non-trivial open problem. 
\end{REM}

\subsection{Heavy-ball Method on the Sum of Moreau Envelopes} \label{subsec:Heavy-ball-on-sum-Moreaus}

\begin{PROP}[inertial averaged proximal minimization method] \label{prop:inertial-avrg-prox-min}
  Suppose $\map{f_i}{\R^N}{\eR}$, $i=1,\ldots,M$ are prox-regular functions at $x^*$ for $v^*=0$ with modulus $\lambda_0>0$ and prox-bounded with threshold $\lambda_f>0$. Let $0<\lambda< \min(\lambda_f, \lambda_0/2)$, $\beta\in\ivalco{0,1}$, and $\alpha \in \ivaloo{0,2(1-\beta)\lambda}$. Suppose that $h = \sum_{i=1}^M \menv\lambda f_i$ has a local minimizer $x^*$ and $H_\ipianodelta$, defined in \eqref{eq:cor:convergence-whole-seq:H-fun}, satisfies \ref{ass:local} and the KL property at $(x^*,x^*)$.
  
  Let $x^0=x^{-1}$ with $x^0\in \R^N$ and let the sequence $\seq[\k\in\N]{x^\k}$ be generated by the following update rule
  \[
     x^\kp \in (1-\alpha\lambda^{-1}) x^\k + \frac{\alpha}{M} \lambda^{-1} \sum_{i=1}^M \prox \lambda {f_i} (x^\k) + \beta (x^\k - x^\km) \,.
  \]
  If $x_0$ is sufficiently close to $x^*$, then sequence $\seq[\k\in\N]{x^\k}$ 
  \begin{itemize}
    \item is uniquely determined,
    \item has the finite length property, 
    \item remains in a neighborhood of $x^*$, 
    \item and converges to a critical point $\tilde x$ of $h$ with $h(\tilde x) = h(x^*)$.
  \end{itemize}
  If all $f_i$, $i=1,\ldots,M$ are proper, lsc, convex, and $\lambda >0$, $\beta\in\ivalco{0,1}$, and $\alpha \in \ivaloo{0,2(1-\beta)\lambda}$, then the sequence has finite length and converges to a global minimizer $\tilde x$ of $h$ for any $x^0\in \R^N$.
\end{PROP}
\begin{proof}
The proof is analogously to that of Proposition~\ref{prop:inertial-prox-min} except for the fact that the Heavy-ball method is applied to $\sum_{i=1}^M \menv\lambda f_i$:
  \[
    \begin{split}
    x^\kp =&\  x^\k - \frac{\alpha}M \sum_{i=1}^M \nabla \menv \lambda {f_i} (x^\k) + \beta(x^\k - x^\km)  \\
          =&\  x^\k - \frac{\alpha}M \lambda^{-1}\sum_{i=1}^M(x^\k - \prox \lambda {f_i} (x^\k)) + \beta(x^\k - x^\km) \\
          =&\  (1- \alpha \lambda^{-1}) x^\k + \frac{\alpha}{M} \lambda^{-1} \sum_{i=1}^M \prox \lambda {f_i} (x^\k) + \beta (x^\k - x^\km)\,.
    \end{split}
  \]
  Instead of scaling the feasible range of step sizes for $\alpha$, the scaling $\frac 1M$ is included in the update formula. 
\end{proof}
\begin{REM}
  Corollary~\ref{cor:ipiano-conv-rates} provides a list of convergence rates for the method in Proposition~\ref{prop:inertial-avrg-prox-min}.
\end{REM}
\begin{REM}
  In contrast to Proposition~\ref{prop:inertial-prox-min}, the sequence of iterates converges to a point $\tilde x$ for which $\sum_{i=1}^M \nabla \menv\lambda{f_i}(\tilde x) =0$ holds. We cannot directly conclude that $0\in \partial (\sum_i f_i)(\tilde x)$. However, if $\nabla \menv\lambda{f_i}(\tilde x)=0$ for all $i=1,\ldots,M$, then under suitable qualification and regularity conditions (see \cite[Corollary 10.9]{Rock98}), we can conclude that $\tilde x$ is a critical point of $\sum_{i=1}^M f_i$. 
\end{REM}
\begin{EX}[inertial averaged projection method for the semi-algebraic feasibility problem]
The algorithm described in Proposition~\ref{prop:inertial-avrg-prox-min} can be used to solve the semi-algebraic feasibility problem of Example~\ref{ex:local-convergence-feasibility}. The conditions in Example~\ref{ex:local-convergence-feasibility} are satisfied.
\end{EX}

\subsection{iPiano on an Objective Involving a Moreau Envelope}  \label{subsec:ipiano-on-Moreau}

\begin{PROP}[inertial alternating proximal minimization method] \label{prop:inertial-alternating-prox-min}
  Suppose $\map{f}{\R^N}{\eR}$ is prox-regular at $x^*$ for $v^*=0$ with modulus $\lambda_0>0$ and prox-bounded with threshold $\lambda_f>0$. Let $0<\lambda< \min(\lambda_f, \lambda_0/2)$. Moreover, suppose that $\map g{\R^N}{\eR}$ is proper, lsc, and simple. Let $x^0=x^{-1}$ with $x^0\in \R^N$ and let the sequence $\seq[\k\in\N]{x^\k}$ be generated by the following update rule
  \[
     x^\kp \in \prox {\alpha} g\left( (1-\alpha\lambda^{-1}) x^\k + \alpha\lambda^{-1} \prox \lambda f (x^\k) + \beta (x^\k - x^\km)\right) \,.
  \]
  We obtain the following cases of convergence results:
  \begin{enumerate}
    \item Assume that $h = g+ \menv\lambda f$ has a local minimizer $x^*$ and $H_\ipianodelta$, defined in \eqref{eq:cor:convergence-whole-seq:H-fun}, satisfies \ref{ass:local} and the KL property at $(x^*,x^*)$. If $x_0$ is sufficiently close to $x^*$, and $\alpha$, $\beta$ are selected according the property of $g$ in one of the last three rows of Table~\ref{tab:iPiano-param-conv} with $L=\lambda^{-1}$, then the sequence $\seq[\k\in\N]{x^\k}$ 
    \begin{itemize}
      \item has the finite length property, 
      \item remains in a neighborhood of $x^*$, 
      \item and converges to a critical point $\tilde x$ of $h$ with $h(\tilde x) = h(x^*)$.
    \end{itemize}
    \item Assume that $f$ is convex, $h = g+\menv\lambda f$ and $x^*$ is a cluster point of  $\seq[\k\in\N]{x^\k}$. Suppose $H_\ipianodelta$, defined in \eqref{eq:cor:convergence-whole-seq:H-fun}, has the KL property at $(x^*,x^*)$. Then, for any $x_0\in \R^N$, and $\alpha$, $\beta$ selected according the property of $g$ in one of the last three rows of Table~\ref{tab:iPiano-param-conv} with $L=\lambda^{-1}$, the sequence $\seq[\k\in\N]{x^\k}$ 
    \begin{itemize}
      \item has the finite length property, 
      \item and converges to a critical point $\tilde x$ of $h$ with $h(\tilde x) = h(x^*)$.
    \end{itemize}
  \end{enumerate}
  If $g$ is convex, the sequence $\seq[\k\in\N]{x^\k}$ is uniquely determined. 
\end{PROP}
\begin{proof}
The proof follows analogously to that of Proposition~\ref{prop:inertial-prox-min} by, either invoking Proposition~\ref{prop:lipschitz-Moreau-envelope} and Corollary~\ref{cor:local-convergence-iPiano} or Proposition~\ref{prop:lipschitz-Moreau-envelope-convex} and Corollary~\ref{cor:convergence-whole-seq}.
\end{proof}
\begin{REM}
  Corollary~\ref{cor:ipiano-conv-rates} provides a list of convergence rates for the method in Proposition~\ref{prop:inertial-alternating-prox-min}.
\end{REM}
\begin{EX}[inertial alternating projection for the semi-algebraic feasibility problem] \label{ex:global-conv-alt-proj}\
\begin{itemize}
  \item The algorithm described in Proposition~\ref{prop:inertial-alternating-prox-min} can be used to solve the semi-algebraic feasibility problem of Example~\ref{ex:local-convergence-feasibility} with $M=2$. The conditions in Example~\ref{ex:local-convergence-feasibility} are satisfied.
  \item If $S_1$ is non-convex and $S_2$ is convex, then the second case of Proposition~\ref{prop:inertial-alternating-prox-min} yields a \emph{globally convergent relaxed alternating projection method} with $g=\ind{S_1}$ and $f=\ind{S_2}$. Table~\ref{tab:iPiano-param-conv} requires the step size conditions $\beta \in \ivalco{0,\frac 12}$ and $\alpha \in \ivaloo{0,1-2\beta}$ (note that $\lambda = 1$), which for $\beta=0$ yields $\alpha \in \ivaloo{0,1}$, which leads to the following update step:
\[
  x^\kp \in \proj {S_1}{(1-\alpha) x^\k + \alpha \,\proj{S_2}{x^\k} }
\]
\end{itemize}
\end{EX}
\begin{EX} 
The algorithm described in Proposition~\ref{prop:inertial-alternating-prox-min} can be used to solve a relaxed version of the following problem:
\[
  \min_{x_1,\ldots,x_M\in\R^N}\, \sum_{i=1}^M g_i(x_i)\,, \quad \st\ x_1=\ldots = x_M \,,
\]
where the convex constraint is replaced by the associated distance function. The functions $\map{g_i}{\R^N}{\eR}$, $i=1,\ldots,M$, $M\in\N$, are assumed to be proper, lsc, simple, and $x=(x_1,\ldots,x_M)\in \R^{N\times M}$ is the optimization variable. This problem belongs to case (ii) of Proposition~\ref{prop:inertial-alternating-prox-min}, i.e. the sequence generated by the inertial alternating proximal minimization method converges globally to a critical point $x^*$ of $\sum_{i=1}^M g_i(x_i) + \frac 12 (\dist(x,C))^2$ where $C:=\set{(x_1,\ldots,x_M)\in \R^{N\times M}\setsep x_1=\ldots = x_M}$. The proximal mapping of $\frac 12 (\dist(x,C))^2$ is the projection onto $C$, which is a simple averaging of $x_1,\ldots,x_M$.
\end{EX}

\subsection{Application: A Feasibility Problem} \label{sec:experiment-feasibility}

\renewcommand{\AA}{\mathcal A}
\newcommand{\setAA}{\mathscr A}
\newcommand{\setRR}{\mathscr R}

We consider the example from \cite{LM08} that demonstrates (local) linear convergence of the alternating projection method. 
The goal is to find an $\dimN\times\dimM$ matrix $X$ of rank $\dimR$ that satisfies a linear system of equations $\AA(X) = B$, i.e.,
\[
  \text{find}\quad X\quad \text{in}\quad \underbrace{\set{X\in \R^{\dimN\times\dimM}\setsep \AA(X) = B}}_{=:\setAA} \cap \underbrace{\set{X\in \R^{\dimN\times\dimM}\setsep \matrank(X) = \dimR}}_{=: \setRR} \,,
\]
where $\map{\AA}{\R^{\dimN\times\dimM}}{\R^\dimD}$ is a linear mapping and $B\in \R^\dimD$. Such feasibility problems are well suited for split projection methods, as the projection onto each set might be easy to conduct. The projections are given  by 
\[
  \proj{\setAA}X = X - \AA^*(\AA \AA^*)^{-1} (\AA(X) - B) \quad \text{and}\quad 
  \proj{\setRR}X = \sum_{i=1}^\dimR \sigma_i u_i v_i^\top \,,
\]
where $USV^\top$ is the singular value decomposition of $X$ with $U=(u_1,u_2,\ldots,u_\dimN)$, $V=(v_1,v_2,\ldots, v_\dimM)$ and singular values $\sigma_1\geq \sigma_2\geq \ldots \geq \sigma_\dimN$ sorted in decreasing order along the diagonal of $S$. Note that the set of rank-$R$ matrices is a $C^2$-smooth manifold \cite[Example 8.14]{Lee2003}, hence prox-regular \cite[Proposition 13.33]{Rock98}. 

We perform the same experiment as in \cite{LM08}, i.e. we randomly generate an operator $\AA$ by constructing random matrices $A_1$, \ldots, $A_D$ and setting $\AA(X) = (\scal{A_1}{X},\ldots,\scal{A_\dimD}{X})$, $\scal{A_i}{X}:=\tr(A^\top X)$, selecting $B$ such that $\AA(X)=B$ has a rank $\dimR$ solution, and the dimensions are chosen as  $\dimM=110$, $\dimN=100$, $\dimR =4$, $\dimD=450$. The performance is measured w.r.t. $\vnorm{\AA(X)-B}$ where $X$ is the result of the projection onto $\setRR$ in the current iteration. 
\begin{table}[t]
  \begin{center}
  \resizebox{\linewidth}{!}{
  \setlength{\tabcolsep}{2pt}
  \begin{tabular}{|c||c|c|c|c|c|c||c|c|c|c|c|c||c|c|c|c|c|c|}
  \hline
  \multicolumn{1}{|r||}{Precision $10^{p}$ $\rightarrow$}
  & $-2$   & $-4$   & $-6$   & $-8$   & $-10$  & $-12$   
  & $-2$   & $-4$   & $-6$   & $-8$   & $-10$  & $-12$   
  & $-2$   & $-4$   & $-6$   & $-8$   & $-10$  & $-12$   \\
  \hline
  Method
  & \multicolumn{6}{c||}{iterations} & \multicolumn{6}{c||}{time [sec]} & \multicolumn{6}{c|}{success [$\%$]} \\
  \hline\hline
  \texttt{alternating projection}
    & $235$ & $886$ &    ---   &    ---   &    ---   &    ---    
    & $ 1.88$  & $ 7.03$  &    ---   &    ---   &    ---   &    ---    
    & $ 100  $ & $97.5  $ & $   0  $ & $   0  $ & $   0  $ & $   0  $  \\
  \hline
  \texttt{averaged projection}
    & $639$ &    ---   &    ---   &    ---   &    ---   &    ---    
    & $ 5.13$  &    ---   &    ---   &    ---   &    ---   &    ---    
    & $ 100  $ & $   0  $ & $   0  $ & $   0  $ & $   0  $ & $   0  $  \\
  \hline
  \texttt{Douglas-Rachford}
    & $974$ &    ---   &    ---   &    ---   &    ---   &    ---    
    & $ 8.10$  &    ---   &    ---   &    ---   &    ---   &    ---    
    & $   2  $ & $   0  $ & $   0  $ & $   0  $ & $   0  $ & $   0  $  \\
  \hline
  \texttt{Douglas-Rachford 75}
    & $209$ & $449$ & $696$ & $949$ &    ---   &    ---    
    & $ 1.68$  & $ 3.62$  & $ 5.63$  & $ 7.66$  &    ---   &    ---    
    & $ 100  $ & $ 100  $ & $ 100  $ & $ 100  $ & $   0  $ & $   0  $  \\
  \hline
  \texttt{glob-altproj, $\alpha = 0.99$}
    & $238$ & $894$ &    ---   &    ---   &    ---   &    ---    
    & $ 1.92$  & $ 7.18$  &    ---   &    ---   &    ---   &    ---    
    & $ 100  $ & $96.5  $ & $   0  $ & $   0  $ & $   0  $ & $   0  $  \\
  \hline
  \texttt{glob-ipiano-altproj, $\beta = 0.45$}
    &    ---   &    ---   &    ---   &    ---   &    ---   &    ---    
    &    ---   &    ---   &    ---   &    ---   &    ---   &    ---    
    & $   0  $ & $   0  $ & $   0  $ & $   0  $ & $   0  $ & $   0  $  \\
  \hline
  \texttt{glob-ipiano-altproj-bt, $\beta = 0.45$}
    & $ 45$ & $ 69$ & $ 90$ & $115$ & $140$ & $166$  
    & $ 0.65$  & $ 1.03$  & $ 1.52$  & $ 2.08$  & $ 2.63$  & $ 3.20$   
    & $ 100  $ & $ 100  $ & $ 100  $ & $ 100  $ & $ 100  $ & $ 100  $  \\
  \hline
  \texttt{heur-ipiano-altproj, $\beta = 0.75$}
    & $ 59$ & $212$ & $386$ & $567$ & $749$ & $925$  
    & $ 0.79$  & $ 2.82$  & $ 5.14$  & $ 7.52$  & $ 9.93$  & $12.22$   
    & $ 100  $ & $ 100  $ & $ 100  $ & $ 100  $ & $ 100  $ & $  91  $  \\
  \hline
  \texttt{loc-heavyball-avrgproj-bt, $\beta = 0.75$}
    & $126$ & $297$ & $502$ & $717$ & $929$ &    ---    
    & $ 2.29$  & $ 5.47$  & $ 9.24$  & $13.21$  & $17.17$  &    ---    
    & $ 100  $ & $ 100  $ & $ 100  $ & $ 100  $ & $93.5  $ & $   0  $  \\
  \hline
  \texttt{loc-ipiano-altproj-bt, $\beta = 0.75$}
    & $ 66$ & $101$ & $138$ & $176$ & $214$ & $252$  
    & $ 1.32$  & $ 2.06$  & $ 2.80$  & $ 3.56$  & $ 4.31$  & $ 5.06$   
    & $ 100  $ & $ 100  $ & $ 100  $ & $ 100  $ & $ 100  $ & $ 100  $  \\
  \hline
\end{tabular}
}
\end{center}
  \caption{\label{tab:alt-avg-proj-convergence-plot}Convergence results for 200 randomly generated feasibility problem as described in Section~\ref{sec:experiment-feasibility}. The table entries show the average number of iterations and the average time that each method requires to reach a certain precision in $\set{10^{-2},10^{-4},\ldots,10^{-12}}$. A dash (\enquote{---}) means that the maximum of 1000 iterations was exceeded. The rightmost part of the table lists the success rate of achieving a certain accuracy within 1000 iterations. For a representative example, the convergence is plotted in Figure~\ref{fig:alt-avg-proj-convergence-plot}.}
\end{table}
\begin{figure}[t]
  \begin{center}
  \centering
    \begin{tikzpicture}
      \begin{axis}[%
        width=0.48\linewidth,%
        height=6.5cm,%
        xmin=0,xmax=1500,%
        xlabel=iterations $\k$,%
        ylabel=$\log_{10}\left(\ds\frac{\vnorm{\AA(X^\k)-B}}{\vnorm{\AA(X^0)-B}}\right)$,%
        legend columns=2,
        legend cell align=left,
        legend style={at={(1.0,1.02)},anchor=south,font=\footnotesize,column sep=10pt}
        ]
      \addplot[thick,blue!80!white,solid,mark=triangle,mark repeat={50},mark size={2},mark options={solid}] %
          table {\codePATH/altproj.dat};
          \addlegendentry{alternating projection};
      \addplot[thick,yellow!80!green,solid,mark=triangle,mark repeat={50},mark size={2},mark options={solid}] %
          table {\codePATH/glob-altproj.dat};
          \addlegendentry{glob-altproj, $\alpha=0.99$};
      \addplot[thick,red!80!black,solid,mark=star,mark repeat={50},mark size={2},mark options={solid}] 
          table {\codePATH/glob-ipiano-altproj-bt-0_45.dat};
          \addlegendentry{glob-ipiano-altproj-bt, $\beta=0.45$};
      \addplot[thick,cyan!80,solid,mark=diamond,mark repeat={50},mark size={2},mark options={solid}] 
          table {\codePATH/glob-ipiano-altproj-0_45.dat};
          \addlegendentry{glob-ipiano-altproj, $\beta=0.45$};
      \addplot[thick,orange!70!cyan,solid,mark=otimes,mark repeat={50},mark size={1.5},mark options={solid}] 
          table {\codePATH/loc-ipiano-altproj-bt-0_75.dat};
          \addlegendentry{loc-ipiano-altproj-bt, $\beta=0.75$};
      \addplot[thick,blue!60!black,solid,mark=triangle*,mark repeat={50},mark size={1.5},mark options={solid}] 
        table {\codePATH/heur-ipiano-altproj-0_75.dat};
        \addlegendentry{heur-ipiano-altproj, $\beta=0.75$};
      \addplot[thick,green!70!black,solid,mark=square,mark repeat={50},mark size={2},mark options={solid}] 
          table {\codePATH/avrgproj.dat};
          \addlegendentry{averaged projection};
      \addplot[thick,orange,solid,mark=square*,mark repeat={50},mark size={1.5},mark options={solid}] 
        table {\codePATH/loc-heavyball-avrgproj-bt-0_75.dat};
        \addlegendentry{loc-heavyball-avrgproj-bt, $\beta=0.75$};
      \addplot[thick,cyan!60!orange,solid,mark=diamond*,mark repeat={50},mark size={1.5},mark options={solid}] 
        table {\codePATH/dr.dat};
        \addlegendentry{Douglas-Rachford};
      \addplot[thick,violet!80!black,solid,mark=otimes*,mark repeat={50},mark size={1.5},mark options={solid}] 
        table {\codePATH/heur-dr-0075.dat};
        \addlegendentry{Douglas-Rachford 75};
      \end{axis}
      \begin{scope}[xshift=0.45\linewidth]
      \begin{axis}[%
        width=0.48\linewidth,%
        height=6.5cm,%
        xmin=0,xmax=30,%
        xlabel=time in seconds%
        ]
      \addplot[thick,blue!80!white,solid,mark=triangle,mark repeat={50},mark size={2},mark options={solid}] %
          table {\codePATH/altproj-timing.dat};
      \addplot[thick,yellow!80!green,solid,mark=triangle,mark repeat={50},mark size={2},mark options={solid}] %
          table {\codePATH/glob-altproj-timing.dat};
      \addplot[thick,red!80!black,solid,mark=star,mark repeat={50},mark size={2},mark options={solid}] 
          table {\codePATH/glob-ipiano-altproj-bt-0_45-timing.dat};
      \addplot[thick,cyan!80,solid,mark=diamond,mark repeat={50},mark size={2},mark options={solid}] 
          table {\codePATH/glob-ipiano-altproj-0_45-timing.dat};
      \addplot[thick,orange!70!cyan,solid,mark=otimes,mark repeat={50},mark size={1.5},mark options={solid}] 
          table {\codePATH/loc-ipiano-altproj-bt-0_75-timing.dat};
      \addplot[thick,blue!60!black,solid,mark=triangle*,mark repeat={50},mark size={1.5},mark options={solid}] 
          table {\codePATH/heur-ipiano-altproj-0_75-timing.dat};
      \addplot[thick,green!70!black,solid,mark=square,mark repeat={50},mark size={2},mark options={solid}] 
          table {\codePATH/avrgproj-timing.dat};
      \addplot[thick,orange,solid,mark=square*,mark repeat={50},mark size={1.5},mark options={solid}] 
          table {\codePATH/loc-heavyball-avrgproj-bt-0_75-timing.dat};
      \addplot[thick,cyan!60!orange,solid,mark=diamond*,mark repeat={50},mark size={1.5},mark options={solid}] 
          table {\codePATH/dr-timing.dat};
      \addplot[thick,violet!80!black,solid,mark=otimes*,mark repeat={50},mark size={1.5},mark options={solid}] 
          table {\codePATH/heur-dr-0075-timing.dat};
      \end{axis}
      \end{scope}
    \end{tikzpicture} 
  \end{center}
  \caption{\label{fig:alt-avg-proj-convergence-plot}Convergence plots for the feasibility problem in Section~\ref{sec:experiment-feasibility}. The inertial methods developed in this paper significantly outperforms all other methods with respect to the number of iterations (left plot) and the actual computation time (right plot).}
\end{figure}
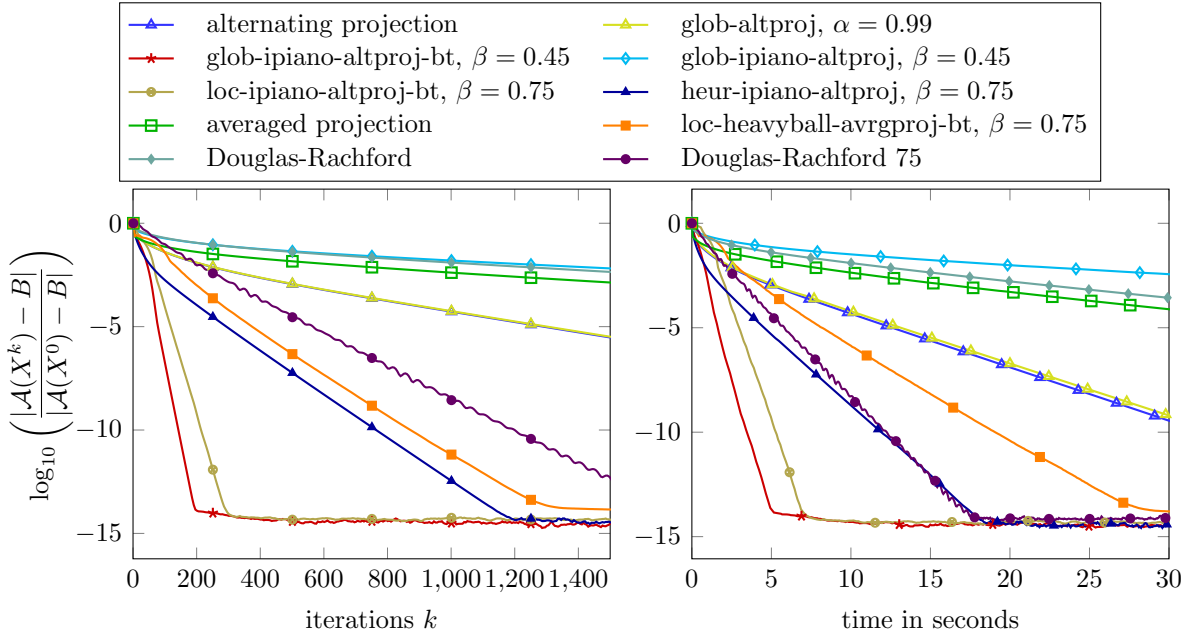

We consider the \emph{alternating projection} method $X^\kp = \proj\setRR{\proj \setAA{X^\k}}$, the \emph{averaged projection} method $X^\kp = \frac{1}{2}\left(\proj\setAA{X^\k} + \proj \setRR{X^\k}\right)$ , the globally convergent relaxed alternating projection method from Example~\ref{ex:global-conv-alt-proj} (\texttt{glob-altproj, $\alpha=0.99$}), and their inertial variants proposed in Sections~\ref{subsec:Heavy-ball-on-sum-Moreaus} and \ref{subsec:ipiano-on-Moreau}. For the Heavy-ball method/inertial averaged projection (\texttt{loc-heavyball-avrgproj-bt, $\beta=0.75$}) in Section~\ref{subsec:Heavy-ball-on-sum-Moreaus} applied to the objective $\dist(X,\setAA)^2 + \dist(X,\setRR)^2$, we use the backtracking line-search version of iPiano \cite[Algorithm 4]{OCBP14} to estimate the Lipschitz constant automatically.  For iPiano/inertial alternating projection (\texttt{glob-ipiano-altproj}) in Section~\ref{subsec:ipiano-on-Moreau} applied to $\min_{X\in\setRR}\,\frac 12 (\dist(X,\setAA))^2$ (i.e. $g$ non-convex, $f$ smooth convex), we use $\beta=0.45\in \ivalco{0,\frac12}$ and $\alpha=0.99(1-2\beta)/L$ with $L=1$, which guarantees global convergence to a stationary point, and a backtracking version (\texttt{glob-ipiano-altproj-bt}) \cite[Algorithm 5]{Ochs15}. Moreover, for the same setting, we use a heuristic version (\texttt{heur-ipiano-altproj, $\beta=0.75$}, theoretically infeasible) with $\alpha$ such that $\alpha\lambda^{-1}=1$ in Proposition~\ref{prop:inertial-alternating-prox-min}. Finally, we also consider the locally convergent version of iPiano in Proposition~\ref{prop:inertial-alternating-prox-min} (\texttt{loc-ipiano-altproj-bt, $\beta=0.75$}) applied to the objective\footnote{The error is measured after projecting the current iterate to the set of rank $R$ matrices.} $\min_{X\in\setAA}\,\frac 12 (\dist(X,\setRR))^2$ (i.e. $g$ convex, $f$ prox-regular, non-convex) with backtracking. For the local convergence results, we assume that we start close enough to a feasible point. Experimentally, all algorithms converge to a feasible point. In theory, backtracking is not required, however as the radius of the neighborhood of attraction is hard to quantify, the algorithm is more stable with backtracking.

We also compare our method against the recently proposed globally convergent Douglas-Rachford splitting for non-convex feasibility problems \cite{LP16a}.  The algorithm depends on a parameter $\gamma$, which in theory is required to be rather small: $\gamma_0:=\sqrt{3/2}-1$. The basic model \texttt{Douglas-Rachford} uses the maximal feasible value for this $\gamma$-parameter. \texttt{Douglas-Rachford 75} is a heuristic version\footnote{The heuristic version of Douglas--Rachford splitting in \cite{LP16a} guarantees boundedness of the iterates. We set $\gamma=150\gamma_0$ and update $\gamma$ by $\max(\gamma/2,0.9999\gamma_0)$ if $\norm{y^\k - y^\km} > t/\k$. We refer to \cite{LP16a} for the meaning of $y^\k$. Since the proposed value $t=1000$ did not work well in our experiment, we optimized $t$ manually. $t=75$ worked best.} proposed in \cite{LP16a}. 

Table~\ref{tab:alt-avg-proj-convergence-plot} compares the methods on a set of 200 randomly generated problems with a maximum of 1000 iterations for each method. Also local methods seem to reliably find a feasible point. This seems to be true also for the heuristic methods \texttt{Douglas-Rachford 75} and \texttt{heur-ipiano-altproj}, which shows that there is still a gap between theory and practice. The inertial algorithms that use backtracking significantly outperform methods without backtracking or inertia. Considering the actual computation time makes this observation even more significant, since backtracking algorithms require to compute the objective value several times per iteration. Interestingly, the globally convergent version of iPiano converged the fastest to a feasible point. The convergence behavior of the methods is visualized in Figure~\ref{fig:alt-avg-proj-convergence-plot} for a representative example.

\section{Conclusions}

In this paper, we proved a local convergence result for abstract descent methods, which is similar to that of Attouch et al. \cite{ABS13}. This local convergence result is applicable to an inertial forward--backward splitting method, called iPiano \cite{OCBP14}. For functions that satisfy the \KL inequality at a local optimum, under a certain growth condition, we verified that the sequence of iterates stays in a neighborhood of a local (or global) minimum and converges to the minimum. As a consequence, the properties that imply convergence of iPiano is required to hold locally only. Combined with a well-known expression for the gradient of Moreau envelopes in terms of the proximal mapping, relations of iPiano to an inertial averaged proximal minimization method and an inertial alternating proximal minimization method are uncovered. These considerations are conducted for functions that are prox-regular instead of the stronger assumption of convexity. For a non-convex feasibility problem, experimentally, iPiano significantly outperforms the alternating projection method and a recently proposed non-convex variant of Douglas--Rachford splitting.

{\small
\bibliographystyle{ieee}
\bibliography{ochs}
}

\end{document}

%% file: mycommands.tex
\newcommand{\R}{\mathbb R}

\newcommand{\N}{\mathbb N}

\newcommand{\set}[2][]{\ifthenelse{\equal{#1}{\empty}}{\lbrace #2 \rbrace}{\lbrace #2\vert\, #1\rbrace}}
\newcommand{\ball}[1]{B_{#1}}

\newcommand{\matrank}{\operatorname{rank}}  

\makeatletter
\def\INT{
  \@ifnextchar[
    {\INT@i}
    {\INT@i[]}
}
\def\INT@i[#1]{
  \@ifnextchar[
    {\INT@ii{#1}}
    {\INT@ii{#1}[]}
}
\def\INT@ii#1[#2]#3#4{
\ensuremath{\int_{#1}^{#2}{#3}\,\mathit d{#4}}
}
\makeatother

\newcommand{\seq}[2][]{(#2)_{#1}}

\newcommand{\rpartial}{\widehat\partial}

\newcommand{\ind}[1]{\delta_{#1}}
\newcommand{\eR}{\overline{\R}}
\newcommand{\Graph}{\operatorname{Graph}}
\newcommand{\dom}{\operatorname{dom}}
\newcommand{\epi}{\operatorname{epi}}

\newcommand{\abs}[1]{\vert #1 \vert}
\newcommand{\norm}[2][]{\Vert #2 \Vert_{#1}}
\newcommand{\vnorm}[2][]{\vert #2 \vert_{#1}}

\newcommand{\scal}[2]{\left\langle #1,#2 \right\rangle}
\newcommand{\map}[3]{#1\colon #2 \to #3}
\newcommand{\smap}[3]{#1\colon #2 \rightrightarrows #3}
\newcommand{\fto}{\overset{f}{\to}}

\newcommand{\proj}[2]{\mathrm{proj}_{#1}(#2)}

\newcommand{\st}{s.t.\xspace}

\newcommand{\dist}{\operatorname{dist}}
\newcommand{\menv}[2]{\ensuremath{e_{#1}{#2}}}

\newcommand{\prox}[2]{\ensuremath{P_{#1}#2}}

\newcommand{\spSob}[2][]{\ensuremath{H^{#2\ifthenelse{\equal{#1}{\empty}}{}{,#1}}}}

\usepackage{mathrsfs}

\newcommand{\spC}[1]{\ensuremath{C^{#1}}}

\newcommand{\ds}{\displaystyle}

\let\phitmp\varphi
\let\varphi\phi
\let\phi\phitmp
\newcommand{\eps}{\varepsilon}

\newcommand{\enquote}[1]{``#1''}

\renewcommand{\o}{\mathrm{o}}

\definecolor{NoteColor}{rgb}{1,0.4,0}
\definecolor{SkyRed}{rgb}{.56,.06,.15}
\definecolor{dSkyBlue}{rgb}{.11,.2,.62}
\definecolor{bSkyBlue}{rgb}{.13,.73,.92}

\newcommand{\leadingzero}[1]{\ifnum #1<10 0\the#1\else\the#1\fi}
\newcommand{\todayIII}{\leadingzero{\day}/\leadingzero{\month}/\the\year}

\usepackage[capitalise]{cleveref}
\newcounter{THMCTR}
\newtheorem{THM}[THMCTR]{Theorem}
\newtheorem{PROP}[THMCTR]{Proposition}
\newtheorem{DEF}[THMCTR]{Definition}

\newtheorem{LEM}[THMCTR]{Lemma}
\newtheorem{COR}[THMCTR]{Corollary}

\newtheorem{ASS}[THMCTR]{Assumption}
\newtheorem{ALG}{\textbf{\textup{Algorithm}}}
\theoremstyle{remark}
\newcounter{EXCTR}
\newtheorem{REM}[EXCTR]{Remark}
\newtheorem{EX}[EXCTR]{Example}

\crefname{THM}{Theorem}{Theorems}
\crefname{PROP}{Proposition}{Propositions}
\crefname{DEF}{Definition}{Definitions}
\crefname{DEFPROP}{DECIDE DEF OR PROP}{DECIDE DEF OR PROP}
\crefname{LEM}{Lemma}{Lemmas}
\crefname{COR}{Corollary}{Corollaries}
\crefname{REM}{Remark}{Remarks}
\crefname{PARA}{Paraphrase}{Paraphrases}
\crefname{ASS}{Assumption}{Assumptions}
\crefname{EX}{Example}{Examples}
\crefname{ALG}{Algorithm}{Algorithms}

\crefname{part}{Part}{Parts}
\crefname{chapter}{Chapter}{Chapters}
\crefname{section}{Section}{Sections}
\crefname{subsection}{Section}{Sections}
\crefname{subsubsection}{Section}{Sections}
\crefname{figure}{Figure}{Figures}
\crefname{table}{Table}{Tables}

\newcommand{\MARK}[2][]{%
\color{SkyRed}#2\ifthenelse{\equal{#1}{\empty}}{}{%
  \marginpar{\textcolor{dSkyBlue}{\ensuremath{\clubsuit} \large \textsc TODO}\\\color{dSkyBlue}\footnotesize#1\hspace*{\fill}\\\color{black}}%
  \textcolor{SkyRed}{\ensuremath{{}^\clubsuit}}%
  }\color{black}\xspace%
}

\newcommand{\tr}{\mathrm{trace}} 


%% file: iAM.bbl
\begin{thebibliography}{10}\itemsep=-1pt

\bibitem{AMA05}
P.~Absil, R.~Mahony, and B.~Andrews.
\newblock Convergence of the iterates of descent methods for analytic cost
  functions.
\newblock {\em SIAM Journal on Optimization}, 16(2):531--547, Jan. 2005.

\bibitem{AB09}
H.~Attouch and J.~Bolte.
\newblock On the convergence of the proximal algorithm for nonsmooth functions
  involving analytic features.
\newblock {\em Mathematical Programming}, 116(1):5--16, June 2009.

\bibitem{ABRS10}
H.~Attouch, J.~Bolte, P.~Redont, and A.~Soubeyran.
\newblock Proximal alternating minimization and projection methods for
  nonconvex problems: An approach based on the {K}urdyka-{{\L}}ojasiewicz
  inequality.
\newblock {\em Mathematics of Operations Research}, 35(2):438--457, May 2010.

\bibitem{ABS13}
H.~Attouch, J.~Bolte, and B.~Svaiter.
\newblock Convergence of descent methods for semi-algebraic and tame problems:
  proximal algorithms, forward--backward splitting, and regularized
  {G}auss--{S}eidel methods.
\newblock {\em Mathematical Programming}, 137(1-2):91--129, 2013.

\bibitem{BC11}
H.~H. Bauschke and P.~L. Combettes.
\newblock {\em Convex analysis and monotone operator theory in Hilbert spaces}.
\newblock Springer, 2011.

\bibitem{BS15}
G.~C. Bento and A.~Soubeyran.
\newblock A generalized inexact proximal point method for nonsmooth functions
  that satisfy the {K}urdyka--{{\L}}ojasiewicz inequality.
\newblock {\em Set-Valued and Variational Analysis}, 23(3):501--517, Feb. 2015.

\bibitem{BCR98}
J.~Bochnak, M.~Coste, and M.-F. Roy.
\newblock {\em Real Algebraic Geometry}.
\newblock Springer, 1998.

\bibitem{BDL06}
J.~Bolte, A.~Daniilidis, and A.~Lewis.
\newblock The {{\L}}ojasiewicz inequality for nonsmooth subanalytic functions
  with applications to subgradient dynamical systems.
\newblock {\em SIAM Journal on Optimization}, 17(4):1205--1223, Dec. 2006.

\bibitem{BDL06b}
J.~Bolte, A.~Daniilidis, and A.~Lewis.
\newblock A nonsmooth {M}orse--{S}ard theorem for subanalytic functions.
\newblock {\em Journal of Mathematical Analysis and Applications},
  321(2):729--740, 2006.

\bibitem{BDLS07}
J.~Bolte, A.~Daniilidis, A.~Lewis, and M.~Shiota.
\newblock Clarke subgradients of stratifiable functions.
\newblock {\em SIAM Journal on Optimization}, 18(2):556--572, 2007.

\bibitem{BDLM10}
J.~Bolte, A.~Daniilidis, A.~Ley, and L.~Mazet.
\newblock Characterizations of {{\L}}ojasiewicz inequalities: subgradient
  flows, talweg, convexity.
\newblock {\em Transactions of the American Mathematical Society},
  362:3319--3363, 2010.

\bibitem{BNPS16}
J.~Bolte, T.~Nguyen, J.~Peypouquet, and B.~Suter.
\newblock From error bounds to the complexity of first-order descent methods
  for convex functions.
\newblock {\em Mathematical Programming}, pages 1--37, Nov. 2016.

\bibitem{BP16}
J.~Bolte and E.~Pauwels.
\newblock Majorization--minimization procedures and convergence of {SQP}
  methods for semi-algebraic and tame programs.
\newblock {\em Mathematics of Operations Research}, 41(2):442--465, Jan. 2016.

\bibitem{BST14}
J.~Bolte, S.~Sabach, and M.~Teboulle.
\newblock Proximal alternating linearized minimization for nonconvex and
  nonsmooth problems.
\newblock {\em Mathematical Programming}, 146(1-2):459--494, 2014.

\bibitem{BLPPR16}
S.~Bonettini, I.~Loris, F.~Porta, M.~Prato, and S.~Rebegoldi.
\newblock On the convergence of variable metric line-search based
  proximal-gradient method under the {Kurdyka}-{Lojasiewicz} inequality.
\newblock {\em arXiv:1605.03791 [math]}, May 2016.

\bibitem{BC16}
R.~I. Bot and E.~R. Csetnek.
\newblock An inertial {Tseng}'s type proximal algorithm for nonsmooth and
  nonconvex optimization problems.
\newblock {\em Journal of Optimization Theory and Applications},
  171(2):600--616, Nov. 2016.

\bibitem{BCL15}
R.~I. Bot, E.~R. Csetnek, and S.~L\'{a}szl\'{o}.
\newblock An inertial forward--backward algorithm for the minimization of the
  sum of two nonconvex functions.
\newblock {\em EURO Journal on Computational Optimization}, 4(1):3--25, Aug.
  2015.

\bibitem{CPR14}
E.~Chouzenoux, J.-C. Pesquet, and A.~Repetti.
\newblock Variable metric forward--backward algorithm for minimizing the sum of
  a differentiable function and a convex function.
\newblock {\em Journal of Optimization Theory and Applications},
  162(1):107--132, July 2014.

\bibitem{CPR16}
E.~Chouzenoux, J.-C. Pesquet, and A.~Repetti.
\newblock A block coordinate variable metric forward--backward algorithm.
\newblock {\em Journal of Global Optimization}, 66(3):457--485, Nov. 2016.

\bibitem{DLMS08}
A.~Daniilidis, A.~Lewis, J.~Malick, and H.~Sendov.
\newblock Prox-regularity of spectral functions and spectral sets.
\newblock {\em Journal of Convex Analysis}, 15(3):547--560, July 2008.

\bibitem{Dries98}
L.~V. den Dries.
\newblock {\em Tame topology and $\o$-minimal structures}, volume 248 of {\em
  London Mathematical Society Lecture Notes Series}.
\newblock Cambridge University Press, 1998.

\bibitem{FGP14}
P.~Frankel, G.~Garrigos, and J.~Peypouquet.
\newblock Splitting methods with variable metric for
  {Kurdyka}--{{\L}}ojasiewicz functions and general convergence rates.
\newblock {\em Journal of Optimization Theory and Applications},
  165(3):874--900, Sept. 2014.

\bibitem{Hosseini15}
S.~Hosseini.
\newblock Convergence of nonsmooth descent methods via
  {K}urdyka--{{\L}}ojasiewicz inequality on {R}iemannian manifolds.
\newblock Technical Report 1523, Institut f\"{u}r Numerische Simulation,
  Rheinische Friedrich--Wilhelms--Universit\"{a}t Bonn, Bonn, Germany, Nov.
  2015.

\bibitem{JM16}
P.~R. Johnstone and P.~Moulin.
\newblock Convergence rates of inertial splitting schemes for nonconvex
  composite optimization.
\newblock {\em arXiv:1609.03626v1 [cs, math]}, Sept. 2016.

\bibitem{JTZ14}
A.~Jourani, L.~Thibault, and D.~Zagrodny.
\newblock Differential properties of the {Moreau} envelope.
\newblock {\em Journal of Functional Analysis}, 266(3):1185--1237, Feb. 2014.

\bibitem{Kurd98}
K.~Kurdyka.
\newblock On gradients of functions definable in o-minimal structures.
\newblock {\em Annales de l'institut Fourier}, 48(3):769--783, 1998.

\bibitem{Lee2003}
J.~Lee.
\newblock {\em Introduction to Smooth Manifolds}.
\newblock Graduate Texts in Mathematics 218. Springer New York, 2003.

\bibitem{LM08}
A.~Lewis and J.~Malick.
\newblock Alternating projections on manifolds.
\newblock {\em Mathematics of Operations Research}, 33(1):216--234, Feb. 2008.

\bibitem{LLM08}
A.~S. Lewis, D.~R. Luke, and J.~Malick.
\newblock Local linear convergence for alternating and averaged nonconvex
  projections.
\newblock {\em Foundations of Computational Mathematics}, 9(4):485--513, Nov.
  2008.

\bibitem{LP15a}
G.~Li, T.~Liu, and T.~K. Pong.
\newblock Peaceman--{Rachford} splitting for a class of nonconvex optimization
  problems.
\newblock {\em Computational Optimization and Applications}, July 2017.
\newblock Published online.

\bibitem{LMNP16}
G.~Li, B.~Mordukhovich, T.~Nghia, and T.~Pham.
\newblock Error bounds for parametric polynomial systems with applications to
  higher-order stability analysis and convergence rates.
\newblock {\em Mathematical Programming}, pages 1--34, Apr. 2016.

\bibitem{LMP15}
G.~Li, B.~Mordukhovich, and T.~Pham.
\newblock New fractional error bounds for polynomial systems with applications
  to {H\"olderian} stability in optimization and spectral theory of tensors.
\newblock {\em Mathematical Programming}, 153(2):333--362, Nov. 2015.

\bibitem{LP17}
G.~Li and T.~Pong.
\newblock Calculus of the exponent of {K}urdyka--{{\L}}ojasiewicz {I}nequality
  and its applications to linear convergence of first-order methods.
\newblock {\em Foundations of Computational Mathematics}, pages 1--34, Aug.
  2017.

\bibitem{LP15b}
G.~Li and T.~K. Pong.
\newblock Global convergence of splitting methods for nonconvex composite
  optimization.
\newblock {\em SIAM Journal on Optimization}, 25(4):2434--2460, Jan. 2015.

\bibitem{LP16a}
G.~Li and T.~K. Pong.
\newblock Douglas--{Rachford} splitting for nonconvex optimization with
  application to nonconvex feasibility problems.
\newblock {\em Mathematical Programming}, 159(1):371--401, Sept. 2016.

\bibitem{LL15}
H.~Li and Z.~Lin.
\newblock Accelerated proximal gradient method for nonconvex programming.
\newblock In C.~Cortes, N.~D. Lawrence, D.~D. Lee, M.~Sugiyama, and R.~Garnett,
  editors, {\em Advances in Neural Information Processing Systems (NIPS)},
  pages 379--387. Curran Associates, Inc., 2015.

\bibitem{LFP16}
J.~Liang, J.~Fadili, and G.~Peyr\'{e}.
\newblock A multi-step inertial forward--backward splitting method for
  non-convex optimization.
\newblock {\em arXiv:1606.02118 [math]}, June 2016.

\bibitem{Loj63}
S.~{\L}ojasiewicz.
\newblock Une propri\'et\'e topologique des sous-ensembles analytiques r\'eels.
\newblock In {\em {L}es {\'E}quations aux {D}\'eriv\'ees {P}artielles}, pages
  87--89, Paris, 1963. \'Editions du centre National de la Recherche
  Scientifique.

\bibitem{Loj93}
S.~{\L}ojasiewicz.
\newblock Sur la g\'eom\'etrie semi- et sous- analytique.
\newblock {\em Annales de l'institut Fourier}, 43(5):1575--1595, 1993.

\bibitem{MP10}
B.~Merlet and M.~Pierre.
\newblock Convergence to equilibrium for the backward {Euler} scheme and
  applications.
\newblock {\em Communications on Pure and Applied Analysis}, 9(3):685--702,
  Jan. 2010.

\bibitem{Noll13}
D.~Noll.
\newblock Convergence of non-smooth descent methods using the
  {Kurdyka}--{{\L}ojasiewicz} inequality.
\newblock {\em Journal of Optimization Theory and Applications},
  160(2):553--572, Sept. 2013.

\bibitem{Ochs15}
P.~Ochs.
\newblock {\em Long term motion analysis for object level grouping and
  nonsmooth optimization methods}.
\newblock PhD thesis, Albert--Ludwigs--Universit{\"a}t Freiburg, Mar 2015.

\bibitem{Ochs16}
P.~Ochs.
\newblock Unifying abstract inexact convergence theorems for descent methods
  and block coordinate variable metric {iPiano}.
\newblock {\em arXiv:1602.07283 [math]}, Feb. 2016.

\bibitem{OCBP14}
P.~Ochs, Y.~Chen, T.~Brox, and T.~Pock.
\newblock i{P}iano: Inertial proximal algorithm for non-convex optimization.
\newblock {\em SIAM Journal on Imaging Sciences}, 7(2):1388--1419, 2014.

\bibitem{ODBP15}
P.~Ochs, A.~Dosovitskiy, T.~Brox, and T.~Pock.
\newblock On iteratively reweighted algorithms for nonsmooth nonconvex
  optimization in computer vision.
\newblock {\em SIAM Journal on Imaging Sciences}, 8(1):331--372, 2015.

\bibitem{PRT00}
R.~Poliquin, R.~Rockafellar, and L.~Thibault.
\newblock Local differentiability of distance functions.
\newblock {\em Transactions of the American Mathematical Society},
  352(11):5231--5249, 2000.

\bibitem{Poliquin91}
R.~A. Poliquin.
\newblock Integration of subdifferentials of nonconvex functions.
\newblock {\em Nonlinear Analysis: Theory, Methods \& Applications},
  17(4):385--398, Jan. 1991.

\bibitem{PR96}
R.~A. Poliquin and R.~T. Rockafellar.
\newblock Prox-regular functions in variational analysis.
\newblock {\em Transactions of the American Mathematical Society},
  348(5):1805--1838, 1996.

\bibitem{Polyak64}
B.~T. Polyak.
\newblock Some methods of speeding up the convergence of iteration methods.
\newblock {\em {USSR} Computational Mathematics and Mathematical Physics},
  4(5):1--17, 1964.

\bibitem{Rock98}
R.~T. Rockafellar and R.~J.-B. Wets.
\newblock {\em Variational Analysis}, volume 317.
\newblock Springer Berlin Heidelberg, Heidelberg, 1998.

\bibitem{STP17}
L.~Stella, A.~Themelis, and P.~Patrinos.
\newblock Forward--backward quasi-{Newton} methods for nonsmooth optimization
  problems.
\newblock {\em Computational Optimization and Applications}, 67(3):443--487,
  July 2017.

\bibitem{XY13}
Y.~Xu and W.~Yin.
\newblock A block coordinate descent method for regularized multiconvex
  optimization with applications to nonnegative tensor factorization and
  completion.
\newblock {\em SIAM Journal on Imaging Sciences}, 6(3):1758--1789, Jan. 2013.

\bibitem{XY17}
Y.~Xu and W.~Yin.
\newblock A globally convergent algorithm for nonconvex optimization based on
  block coordinate update.
\newblock {\em Journal of Scientific Computing}, pages 1--35, Feb. 2017.

\bibitem{ZK93}
S.~Zavriev and F.~Kostyuk.
\newblock Heavy-ball method in nonconvex optimization problems.
\newblock {\em Computational Mathematics and Modeling}, 4(4):336--341, 1993.

\end{thebibliography}
